\documentclass[11pt,a4paper]{article}
\usepackage[letterpaper,top=2cm,bottom=2cm,left=3cm,right=3cm,marginparwidth=1.75cm]{geometry}
\usepackage{amsfonts,amsgen,amstext,amsbsy,amsopn,amsfonts,amssymb,amscd}
\usepackage[leqno]{amsmath}
\usepackage[amsmath,amsthm,thmmarks]{ntheorem}
\usepackage{epsf,epsfig}
\usepackage{float}
\usepackage{dsfont}
\usepackage{ebezier,eepic}
\usepackage{color}
\usepackage{tikz}
\usepackage{multirow}
\usepackage{mathrsfs}
\usepackage{graphicx}
\usepackage{subfigure}
\newtheorem{definition}{Definition} [section]
\newtheorem{theorem}[definition]{Theorem}
\newtheorem{lemma}[definition]{Lemma}
\newtheorem{proposition}[definition]{Proposition}

\newtheorem{corollary}[definition]{Corollary}
\newtheorem{conjecture}[definition]{Conjecture}
\newtheorem{claim}[definition]{Claim}

\newtheorem{fact}[definition]{Fact}

\makeatletter \@addtoreset{equation}{section}\makeatother

\def\hf{\mathcal{F}}
\def\hg{\mathcal{G}}
\def\hh{\mathcal{H}}
\def\thf{\tilde{\mathcal{F}}}
\def\ha{\mathcal{A}}
\def\hb{\mathcal{B}}
\def\hs{\mathcal{S}}
\def\hl{\mathcal{L}}
\def\hk{\mathcal{K}}
\def\hp{\mathcal{P}}
\def\hr{\mathcal{R}}

\title{The exact Erd\H{o}s-Ko-Rado Theorem for 3-wise $t$-intersecting uniform families}
\author{Peter Frankl\footnote{R\'{e}nyi Institute, Budapest, Hungary. Email: \texttt{frankl.peter@renyi.hu}.}\quad\quad
Jian Wang \footnote{School of Mathematics, Sichuan University, Chengdu, 610065, China. Research supported by National Natural Science Foundation of China No. 12622121. Email: \texttt{wangjianmath01@scu.edu.cn}.}
}
\date{}

\begin{document}

\maketitle

\begin{abstract}
One of the central problems in extremal set theory is the determination of the maximum size of an $r$-wise $t$-intersecting $k$-graph on $n$ vertices. For $r=2$ after decades of effort, a complete solution was given by the Complete Intersection Theorem (\cite{AK0}). However the case $r\geq 3$ remains wide open. The $t$-star consisting of all $k$-sets containing a fixed $t$-set provides a lower bound. For $r=3$ and $k\geq t+2\geq44$, we determine the exact threshold for which it is optimal for $n>k$. The condition is essentially $n\geq \sqrt{t}k$. For $k\geq t+2\geq47$ and $n\geq\frac{\sqrt{4t+9}-1}{2}k$, we also establish the corresponding Hilton-Milner-type stability result.
\end{abstract}

\section{Introduction}

Let $[n]$ be the standard $n$-element set $\{1,2,\ldots,n\}$. Let $2^{[n]}$ denote the power set of $[n]$ and let $\binom{[n]}{k}$ denote the family of all $k$-element subsets of $[n]$.

\begin{definition}
A family $\hf\subset 2^{[n]}$ is called $r$-wise $t$-intersecting if $|F_1\cap F_2\cap \ldots\cap F_r|\geq t$ for all $F_1,F_2,\ldots,F_r\in \hf$, $r\geq 2$, $t\geq 1$.
\end{definition}

For a fixed $t$-set $T$, the full $t$-star $\hs_T=\{S\subset [n]\colon T\subset S\}$ is the simplest construction for a large family that is $r$-wise $t$-intersecting.

If $\hf$ is $r$-wise $t$-intersecting but $\cap \hf: =\cap \{F\colon F\in \hf\}$ has size less than $t$, then $\hf$ is called {\it non-trivial}.

For $n\geq s$, define $\ha(n,s)=\{A\subset [n]\colon |A\cap [s]|\geq s-1 \}$. It is easy to verify that $\ha(n,r+t)$ is $r$-wise $t$-intersecting and non-trivial.

In the present paper we deal with uniform families, i.e., $\hf\subset \binom{[n]}{k}$. For $n>k\geq s-1$, define
\[
\ha(n,k,s) =\ha(n,s) \cap \binom{[n]}{k}.
\]

A simple calculation shows that
\[
|\ha(n,k,t+2)| \gtreqqless \binom{n-t}{k-t} \mbox{ is equivalent to } n  \lesseqqgtr (t+1)(k-t+1).
\]

The next result shows the relevance of these inequalities.

\begin{theorem}[Exact Erd\H{o}s-Ko-Rado Theorem, \cite{EKR}, \cite{F78}, \cite{W84}]\label{thm-EKR}
Suppose that  $\mathcal{F}\subset \binom{[n]}{k}$ is 2-wise $t$-intersecting, $k\geq t\geq 1$, $n\geq (t+1)(k-t+1)$. Then
\[
|\mathcal{F}| \leq \binom{n-t}{k-t}.
\]
\end{theorem}

Let us also mention the following result.

\begin{theorem}[\cite{FF91}, \cite{AK0}]\label{thm-1.3}
Suppose that $\hf\subset \binom{[n]}{k}$ is  2-wise $t$-intersecting, $k>t\geq 2$, $\frac{t+3}{2}(k-t+1)\leq n\leq (t+1)(k-t+1)$. Then
\[
|\mathcal{F}| \leq |\ha(n,k,t+2)|.
\]
\end{theorem}

For $n>k\geq s-1\geq2$, define
\begin{align*}
\mathcal{B}(n,k,s):=&\left\{B\in \binom{[n]}{k}\colon [s-2]\subset B,\ B\cap [s-1,k+1]\neq \emptyset\right\}\\[3pt]
&\ \bigcup \left\{[k+1]\setminus \{j\}\colon 1\leq j\leq s-2\right\}.
\end{align*}
It is easy to check that $\mathcal{B}(n,k,t+r)$ is non-trivial and $r$-wise $t$-intersecting. Moreover,
\begin{align*}
&|\ha(n,k,t+r)| =(t+r)\binom{n-t-r}{k-t-r+1}+\binom{n-t-r}{k-t-r}, \\[3pt]
&|\hb(n,k,t+r)| = \binom{n-t-r+2}{k-t-r+2}-\binom{n-k-1}{k-t-r+2}+t+r-2.
\end{align*}

Define
\begin{align*}
& g(n,k,r,t):=\max \left\{|\mathcal{F}|\colon \mathcal{F}\subset \binom{[n]}{k}\mbox{ is $r$-wise $t$-intersecting}\right\}, \\[3pt]
 & h(n,k,r,t) := \max \left\{|\mathcal{F}|\colon \mathcal{F}\subset \binom{[n]}{k}\mbox{ is non-trivial $r$-wise $t$-intersecting}\right\}.
\end{align*}
Theorem \ref{thm-EKR} shows that
\[
g(n,k,2,t)=\binom{n-t}{k-t} \mbox{ for } n\geq (t+1)(k-t+1).
\]

\begin{theorem}[\cite{HM}, \cite{F78-2}, \cite{AK}]\label{thm-HMF}
For $n\geq (t+1)(k-t+1)$ and $k>t\geq1$,
\[
h(n,k,2,t)= \max\left\{|\mathcal{A}(n,k,t+2)|,|\mathcal{B}(n,k,t+2)|\right\}.
\]
\end{theorem}

Let us mention that in view of Theorem \ref{thm-1.3}, for $\frac{t+3}{2}(k-t+1) \leq n\leq (t+1)(k-t+1)$,
\[
h(n,k,2,t)=|\mathcal{A}(n,k,t+2)|.
\]

Let us recall an easy statement.

\begin{fact}[cf. e.g., \cite{F87}]\label{fact-key}
If $r\geq3$, $t\geq1$ and $\mathcal{F}$ is non-trivial $r$-wise $t$-intersecting, then $\mathcal{F}$ is non-trivial $(r-1)$-wise $(t+1)$-intersecting.
\end{fact}

Applying Fact \ref{fact-key} repeatedly yields that  if $\hf$ is non-trivial $r$-wise $t$-intersecting then $\hf$ is also non-trivial 2-wise $(t+r-2)$-intersecting. Thus we have the following corollary of Theorems \ref{thm-EKR} and \ref{thm-HMF}.

\begin{corollary}\label{cor-1.4}
\begin{itemize}
    \item[(i)] For $n>k>t\geq2$,
    \[
    h(n,k,t+1,1) \leq h(n,k,t,2)\leq \ldots\leq h(n,k,2,t).
    \]
    \item[(ii)] For $n\geq (t+r-1)(k-t-r+3)$, $r\geq2$, $t\geq1$ and $k\geq t+r-1$,
\begin{align*}
g(n,k,r,t)= \binom{n-t}{k-t},\quad
h(n,k,r,t)= \max\left\{|\mathcal{A}(n,k,r+t)|,|\mathcal{B}(n,k,r+t)|\right\}.
\end{align*}
\end{itemize}
\end{corollary}

Combining Corollary \ref{cor-1.4} (ii) and Theorem \ref{thm-1.3}, one can deduce

\begin{theorem}[\cite{BL}]\label{thm-1.5}
For $n\geq \frac{(t+r+1)}{2}(k-t-r+3)$, $r\geq2$, $t\geq1$ and $k\geq t+r-1$,
\[
h(n,k,r,t)= \max\left\{|\mathcal{A}(n,k,r+t)|,|\mathcal{B}(n,k,r+t)|\right\}.
\]
\end{theorem}

In \cite{FW25}, the following result was proved.

\begin{theorem}[\cite{FW25}]
For $n\geq \left(2.5 t\right)^{\frac{1}{r-1}}(k-t)+k$ and $k> t\geq 1$,
\[
g(n,k,r,t)= \binom{n-t}{k-t}.
\]
\end{theorem}

Let us mention another recent result.

\begin{theorem}[\cite{FW23}]
    Let $0<\varepsilon <\frac{1}{10}$. For $n\geq \frac{4}{\varepsilon^2}+7$ and $(\frac{1}{2}+\varepsilon)n\leq k\leq \frac{3}{5}n-3$,
    \[
    h(n,k,3,1)= |\mathcal{A}(n,k,4)|.
    \]
\end{theorem}

In the present paper, we mainly consider $g(n,k,3,t)$ and $h(n,k,3,t)$. By Fact \ref{fact-key}, if $\hf\subset \binom{[n]}{k}$ is non-trivial 3-wise $t$-intersecting, then $\hf$ is 2-wise $(t+1)$-intersecting. If $k=t+1$ then $|\hf|\leq 1$ follows. Thus, we only consider the  $k\geq t+2$ case in the rest of the paper.

Our first result is the following.

\begin{theorem}\label{thm-main}
For $n\geq k+1+\sqrt{\frac{t+3}{1.696}}(k-t-1)$ and $k\geq t+2$, $t\geq42$,
\[
g(n,k,3,t) = \max\left\{\binom{n-t}{k-t},|\ha(n,k,t+3)|\right\}.
\]
\end{theorem}

Define
\[
n_0(k,t) := \frac{1}{2}\left(\sqrt{(4t+9)k^2-2(4t^2+11t+3)k+4t^3+13t^2+6t+1}-k+3(t+1)\right).
\]
Direct calculation shows that
\begin{align}\label{ineq-1.0}
|\ha(n,k,t+3)| \gtreqqless \binom{n-t}{k-t} \mbox{ is equivalent to } n  \lesseqqgtr n_0(k,t).
\end{align}

It can be checked that for $k\geq t+2\geq 44$,
\[
n_0(k,t) \geq  k+1+\sqrt{\frac{t+3}{1.696}}(k-t-1).
\]
By \eqref{ineq-1.0} and Theorem \ref{thm-main}, we obtain the Exact Erd\H{o}s-Ko-Rado Theorem for 3-wise $t$-intersecting families.

\begin{theorem}\label{thm-main0}
Let $n>k\geq t+2$ and  $t\geq42$. Then
\[
g(n,k,3,t) = \binom{n-t}{k-t} \mbox{ if and only if } n\geq n_0(k,t).
\]
\end{theorem}

 One can check that $n_0(k,t)$ is smaller than $\frac{\sqrt{4t+9}-1}{2}k$ for $k\geq t+2\geq44$.
Thus Theorem \ref{thm-main0} has the following corollary, which confirms Conjecture 7.2 in \cite{FW25} for $t\geq42$ and $k\geq t+2$.

\begin{corollary}
Let $n\geq \frac{\sqrt{4t+9}-1}{2}k$ and $k\geq t+2\geq44$. Then
\[
g(n,k,3,t) = \binom{n-t}{k-t}.
\]
\end{corollary}

For non-trivial 3-wise $t$-intersecting families, we improve Theorem \ref{thm-1.5} as follows.

\begin{theorem}\label{thm-main2}
  For $n\geq\frac{\sqrt{4t+9}-1}{2}k$, $k\geq t+2$ and $t\geq45$,
\[
h(n,k,3,t)= \max\left\{|\mathcal{A}(n,k,t+3)|,|\mathcal{B}(n,k,t+3)|\right\}.
\]
\end{theorem}

\section{Tools of the Proofs}

Recall the {\it lexicographic order} $A <_{L} B$ for $A,B\in \binom{[n]}{k}$ defined by $A<_L B$ iff $\min\{i\colon i\in A\setminus B\}<\min\{i\colon i\in B\setminus A\}$. For $n>k>0$ and $\binom{n}{k}\geq m>0$ let $\hl(n,k,m)$ denote the family of the first $m$ sets $A\in \binom{[n]}{k}$ in the lexicographic order.

\begin{lemma}[Hilton \cite{Hilton}]
Let $n,a,b$ be positive integers, $n\geq a+b$. Suppose that $\ha\subset \binom{[n]}{a}$ and $\hb\subset \binom{[n]}{b}$ are cross-intersecting. Then $\hl(n,a,|\ha|)$ and $\hl(n,b,|\hb|)$ are cross-intersecting as well.
\end{lemma}

Let us recall the shifting operator invented by Erd\H{o}s, Ko and Rado \cite{EKR}, which has proved to be the most powerful tool in extremal set theory.

For $\hf\subset \binom{[n]}{k}$ and $1\leq i<j\leq n$, define the $(i,j)$-shift $S_{ij}$ by
$$S_{ij}(\hf)=\left\{S_{ij}(F)\colon F\in\hf\right\},$$
where
$$S_{ij}(F)=\left\{
                \begin{array}{ll}
                 F':= (F\setminus\{j\})\cup\{i\}, & \mbox{ if } j\in F, i\notin F \text{ and } F'\notin \hf; \\[5pt]
                  F, & \hbox{otherwise.}
                \end{array}
              \right.
$$

A family $\hf\subset \binom{[n]}{k}$ is called {\it shifted} (or {\it initial}) if $S_{ij}(\hf)=\hf$ holds for all $1\leq i<j\leq n$.

\begin{proposition}[\cite{F87}]\label{prop-2.1}
By  repeated shifting, one can transform an arbitrary $k$-graph into a shifted $k$-graph with the same number of edges.
\end{proposition}

Let us define the {\it shifting partial order} $\prec$. For two $k$-sets $A$ and $B$ where $A=\{a_1,\ldots,a_k\}$, $a_1<\ldots<a_k$ and $B=\{b_1,\ldots,b_k\}$, $b_1<\ldots<b_k$ we say that $A$ {\it precedes} $B$ and denote it by $A\prec B$ if $a_i\leq b_i$ for all $1\leq i\leq k$.

Shifted families have the following nice properties.

\begin{proposition}[\cite{F87}]
If $\hf\subset \binom{[n]}{k}$ is shifted, then  $A\prec B$ and $B\in \hf$ imply $A\in \hf$.
\end{proposition}

\begin{proposition}[\cite{F78}, cf. also \cite{FW252}]\label{prop-F78}
Let $\hf\subset \binom{[n]}{k}$ be a shifted family.
\begin{itemize}
  \item[(i)] If $(1,2,\ldots,p,p+2,p+4,\ldots,2k-p)\notin \hf$ for some $0\leq p\leq k-1$, then
\[
|\hf| \leq \binom{n}{k-p-1}.
\]
  \item[(ii)]
If $(p,p+2,p+4,\ldots,p+2k-2)\notin \hf$ for some $2\leq p\leq k-1$, then
\[
|\hf| \leq \binom{n}{k}-\binom{n-p+2}{k}+\binom{n-p+2}{k-1}.
\]
\end{itemize}
\end{proposition}

We need  the following universal bound for 2-wise $t$-intersecting families.

\begin{lemma}[\cite{F78}]\label{lem-F20}
    Let $\mathcal{F}\subset \binom{[n]}{k}$ be a 2-wise $t$-intersecting family. Then for $n\geq k\geq t\geq 1$ and  $n>2k-t$,
    \[
    |\hf| \leq \binom{n}{k-t}.
    \]
\end{lemma}

We also need the following two results.

\begin{theorem}[Hilton \cite{Hilton}]\label{thm-hilton}
     Let  $\ha_1,\ha_2,\ldots,\ha_m\subset \binom{[n]}{k}$ be pairwise cross-intersecting families. For $n\geq 2k$,
     \[
     \sum_{1\leq i\leq m} |\ha_i| \leq \left\{ \begin{array}{ll}
                m\binom{n-1}{k-1}, & \mbox{ if } m\geq \frac{n}{k}; \\[5pt]
                  \binom{n}{k}, & \hbox{otherwise.}
                \end{array}\right.
     \]
\end{theorem}

\begin{theorem}[Wang--Zhang \cite{WZ}]\label{thm-lz}
     Let  $\ha_1,\ha_2,\ldots,\ha_m\subset \binom{[n]}{k}$ be pairwise cross $t$-intersecting families. For $n\geq (t+1)(k-t+1)$,
     \[
     \sum_{1\leq i\leq m} |\ha_i| \leq \max\left\{\binom{n}{k}, m\binom{n-t}{k-t}\right\}.
     \]
\end{theorem}

For $\hf\subset \binom{[n]}{k}$ and $P\subset Q\subset [n]$, define
\[
\hf(P,Q) := \{F\setminus Q\colon F\in \hf,\ F\cap Q=P\}.
\]
For $P=Q$, we use $\hf(P)$ to denote $\hf(P,Q)$ for short.

\begin{proposition}\label{prop-shift-nontrivial}
 Let $r\geq 3$, $t\geq1$, $k\geq t+r-1$, and let $\mathcal{F}\subset \binom{[n]}{k}$ be a non-trivial $r$-wise $t$-intersecting family of maximum size. Then for $n\geq (r-1)(k-t-r+3)+t$ we may assume that $\mathcal{F}$ is shifted.
\end{proposition}

\begin{proof}
Since the shifting operation preserves the $r$-wise $t$-intersecting property (cf. \cite{F87}), one can apply $S_{ij}$, $1\leq i<j\leq n$, to $\mathcal{F}$ repeatedly whenever $S_{ij}(\mathcal{F})$ is also non-trivial. By Proposition \ref{prop-2.1}, the only possible trouble is that $S_{ij}(\hf)$ becomes a $t$-star.

Suppose that $\mathcal{F}$ is not a $t$-star but $S_{ij}(\mathcal{F})$ is a $t$-star. Let $\mathcal{G}=S_{ij}(\mathcal{F})$. Every element other than $i$ that belongs to all members of $\mathcal G$ also belongs to all members of $\mathcal F$, since a shift can insert only $i$. If $i\notin\cap\mathcal G$, then $\cap\mathcal G\subseteq\cap\mathcal F$, contradicting $|\cap\mathcal G|\geq t>|\cap\mathcal F|$. Hence $i\in\cap\mathcal G$. Now let
\[
\cap \{G\colon G\in \mathcal{G}\} = S\cup \{i\}.
\]
Then $\cap \mathcal{G}(S\cup \{i\})=\emptyset$.
Since $\mathcal{G}$ is a $t$-star,  $|S|\geq t-1$. Note that $S\subset F$ for all $F\in \mathcal{F}$ and $\mathcal{F}$ is not a $t$-star. We infer $|S|=t-1$.

By Fact \ref{fact-key},  $\mathcal{F}$ is $2$-wise $(t+r-2)$-intersecting. Since  the shifting operation also preserves the $2$-wise $(t+r-2)$-intersecting property,  $\mathcal{G}$ is $2$-wise $(t+r-2)$-intersecting. Thus $\mathcal{G}(S\cup \{i\})$ is non-trivial $2$-wise $(r-2)$-intersecting. By Theorem \ref{thm-HMF}, for $n-t\geq (r-2+1)((k-t)-(r-2)+1)=(r-1)(k-t-r+3)$ we have
\begin{align*}
|\mathcal{F}| =|\mathcal{G}(S\cup \{i\})|&\leq \max\left\{|\mathcal{A}(n-t,k-t,r)|,|\mathcal{B}(n-t,k-t,r)|\right\}\\[3pt]
&< \max\left\{|\mathcal{A}(n,k,r+t)|,|\mathcal{B}(n,k,r+t)|\right\}\\[3pt]
&\leq h(n,k,r,t),
\end{align*}
contradicting the fact that $\hf$ is of the maximum size.
Thus one can apply the shifting operator to $\hf$ repeatedly without turning $\hf$ into a $t$-star.
\end{proof}

Most likely the condition $n\geq (r-1)(k-t-r+3)+t$ can be relaxed further. However we do not know how to prove it.

\begin{corollary}\label{cor-2.1}
Let $k\geq t+2$ and $t\geq1$. Let $\mathcal{F}\subset \binom{[n]}{k}$ be a non-trivial $3$-wise $t$-intersecting family with $|\mathcal{F}|=h(n,k,3,t)$. Then for $n\geq 2k-t$ we may assume that $\mathcal{F}$ is shifted.
\end{corollary}

\begin{lemma}\label{lem-2.7}
   Let $s\geq t>0$ and let $\hf\subset 2^{[n]}$ be $r$-wise $t$-intersecting, shifted.  Suppose that $A_1,A_2,\ldots,A_r\subset [s]$ satisfy $|A_1\cap A_2\cap \ldots \cap A_r|<t$, $B_i\in \hf(A_i,[s])$, $i=1,2,\ldots,r$. Then
    \begin{align}\label{ineq-lem2.7-1}
        |B_1\cap B_2\cap \ldots \cap B_r|\geq (r-1)s+t-|A_1|-|A_2|-\cdots-|A_r|.
    \end{align}
\end{lemma}

\begin{proof}
Suppose that $A_i\cup B_i\in \hf$, $1\leq i\leq r$, form a counter-example with $|A_1|+|A_2|+\cdots +|A_r|$ as large as possible.
Note that $|A_1\cap A_2\cap \cdots \cap A_r|\leq t-1$ implies
\begin{align}\label{ineq-lem2.7-2}
    |A_1|+|A_2|+\cdots +|A_r|\leq (r-1)s+t-1.
\end{align}
Should we have equality in \eqref{ineq-lem2.7-2}, \eqref{ineq-lem2.7-1} would reduce to $|B_1\cap B_2\cap \ldots\cap B_r|\geq 1$, which holds evidently by $\hf$ being $r$-wise $t$-intersecting. I.e., $A_i\cup B_i$, $1\leq i\leq r$ is not a counter-example, a contradiction. Thus we may assume that
\begin{align}\label{ineq-lem2.7-3}
    |A_1|+|A_2|+\cdots +|A_r|\leq (r-1)s+t-2.
\end{align}
There are two possibilities. Either $|A_1\cap A_2\cap \ldots\cap A_r|\leq t-2$. In this case fix $x\in [s]\setminus (A_1\cap A_2\cap \ldots \cap A_r)$ arbitrarily and assume by symmetry $x\notin A_r$. Or there exists $x\in [s]$ which is  contained in less than $r-1$ of $A_1,A_2,\ldots,A_r$. Again by symmetry assume $x\notin A_r$. Now setting $A_i'=A_i$ for $1\leq i< r$ and $A_r'=A_r\cup \{x\}$, in either case $|A_1'\cap A_2'\cap \ldots\cap A_r'|<t$ is still valid.

In view of the $r$-wise $t$-intersecting property, $B_1\cap B_2\cap \ldots \cap B_r\neq \emptyset$. Let $y$ be in the intersection and define $B_i'=B_i$ for $1\leq i<r$, $B_r'=B_r\setminus\{y\}$. Note that $A_i'\cup B_i'=A_i\cup B_i\in \hf$ for $1\leq i<r$. Since  $A_r'\cup B_r'\prec A_r\cup B_r$,  by shiftedness  $A_r'\cup B_r'\in \hf$. As
\begin{align*}
|B_1'\cap B_2'\cap \ldots\cap B_r'| =|B_1\cap B_2\cap \ldots\cap B_r| -1&<(r-1)s+t-|A_1|-|A_2|-\cdots-|A_r|-1\\[3pt]
&=(r-1)s+t-|A_1'|-|A_2'|-\cdots-|A_r'|,
\end{align*}
$A_i'\cup B_i'$, $1\leq i\leq r$ form a counter-example with $|A_1'|+|A_2'|+\cdots+|A_r'|=|A_1|+|A_2|+\cdots+|A_r|+1$, which contradicts our assumption.
\end{proof}

\section{Some structural lemmas}

For $\hf\subset \binom{[n]}{k}$, $0\leq i\leq t+3$ and $T\in \binom{[t+3]}{t+3-i}$, define
\[
\hf_i =\{F\in \hf\colon |F\cap [t+3]|=t+3-i\} \mbox{ and } \tilde{\hf}_i(T) =\{E\subset [t+4,n]\colon T\cup E\in \hf_i\}.
\]
Clearly $\hf=\hf_0\cup \hf_1\cup \ldots\cup \hf_{t+3}$.

\begin{lemma}\label{lem-3.1}
Assume $k\geq t+3$ and $t\geq1$.
    Suppose that $\hf\subset \binom{[n]}{k}$ is a shifted non-trivial 3-wise $t$-intersecting family. Then
    \begin{itemize}
        \item[(i)] For $4\leq i\leq t+3$ and any $T\in \binom{[t+3]}{t+3-i}$, $\tilde{\hf}_i(T)$ is 3-wise $(3i-3)$-intersecting.
        \item[(ii)] For any $T\in \binom{[t+3]}{t}$, $\thf_3(T)$ is  2-wise 4-intersecting  and $\thf_3(T)$ is  3-wise 6-intersecting if $T\neq [t]$. Moreover, if $\hf$ is 2-wise $(t+j)$-intersecting for some $1\leq j\leq3$, then $\thf_3(T)$ is  2-wise $(3+j)$-intersecting.
        \item[(iii)]  For any $T\in \binom{[t+3]}{t+1}\setminus \{[t+1]\}$, $\thf_{2}(T)$ is 2-wise 2-intersecting and $\thf_{2}(T)$ is 3-wise 3-intersecting if $[t]\not\subset T$. Moreover, if $\hf$ is 2-wise $(t+j)$-intersecting  for some $1\leq j\leq3$,  then $\thf_2(T)$ is  2-wise $(j+1)$-intersecting for any $T\in \binom{[t+3]}{t+1}\setminus \{[t+1]\}$.
    \end{itemize}
\end{lemma}

\begin{proof}
 Let us show first (i). Let $A_1=A_2=A_3=T$. Then $|A_1\cap A_2\cap A_3|=|T|=t+3-i<t$. Then for any $B_1,B_2,B_3\in \tilde{\hf}_i(T)$,  by Lemma \ref{lem-2.7} we have
 \[
 |B_1\cap B_2\cap B_3| \geq 2(t+3)+t-3|T| =3t+6-3(t+3-i) =3i-3.
 \]
 Thus $\tilde{\hf}_i(T)$ is 3-wise $(3i-3)$-intersecting.

Secondly we show (ii). By Fact \ref{fact-key}, $\hf$ is 2-wise $(t+1)$-intersecting. Let $A_1=A_2=[t]$.  Then $|A_1\cap A_2|=t<t+1$.  For any $B_1,B_2\in \tilde{\hf}_3([t])$,  applying Lemma \ref{lem-2.7} with $(s,r,t)=(t+3,2,t+1)$ we have
\[
|B_1\cap B_2|\geq (t+3)+(t+1)-2t =4.
\]
Thus $\tilde{\hf}_3([t])$ is 2-wise 4-intersecting. In the next application of Lemma \ref{lem-2.7}, its hypothesis $s\geq t$ becomes $t+3\geq t+j$, so we require $j\leq3$. Similarly, if $\hf$ is 2-wise $(t+j)$-intersecting for some $1\leq j\leq3$, then $\thf_3([t])$ is  2-wise $(3+j)$-intersecting. By shiftedness,  $\tilde{\hf}_3(T)\subset \tilde{\hf}_3([t])$ for any $T\in \binom{[t+3]}{t}$. Thus $\thf_3(T)$ is  2-wise $(3+j)$-intersecting for any $T\in \binom{[t+3]}{t}$.

For any $T\in \binom{[t+3]}{t}\setminus \{[t]\}$, let $A_1=A_2=[t]$, $A_3=[t-1]\cup \{t+1\}$.  Then $|A_1\cap A_2\cap A_3|=t-1<t$.  For any $B_1,B_2,B_3\in \tilde{\hf}_3(T)$, by shiftedness $A_1\cup B_1, A_2\cup B_2, A_3\cup B_3\in \hf$.
Applying Lemma \ref{lem-2.7} with $(s,r,t)=(t+3,3,t)$  we have
\[
|B_1\cap B_2\cap B_3| \geq 2(t+3)+t-3t =6.
\]
Thus $\tilde{\hf}_3(T)$ is  3-wise 6-intersecting.

Thirdly we show (iii). Let $T\in \binom{[t+3]}{t+1}\setminus \{[t+1]\}$ and let $A_1=[t+1]$, $A_2=[t]\cup \{t+2\}$. Then for any $B_1,B_2\in \tilde{\hf}_2(T)$,  by shiftedness $A_1\cup B_1$, $A_2\cup B_2\in \hf$. Applying Lemma \ref{lem-2.7} with $(s,r,t)=(t+3,2,t+1)$  we have
\[
|B_1\cap B_2| \geq (t+3)+(t+1)-2(t+1) =2.
\]
Thus $\tilde{\hf}_{2}(T)$ is 2-wise 2-intersecting. Similarly, if $\hf$ is 2-wise $(t+j)$-intersecting for some $1\leq j\leq3$, then $\thf_2(T)$ is  2-wise $(j+1)$-intersecting.

Let $T\in \binom{[t+3]}{t+1}$ with $[t]\not\subset T$ and let $A_1=[t+1]$, $A_2=[t+2]\setminus \{t+1\}$, $A_3=[t+2]\setminus \{t\}$. Then $|A_1\cap A_2\cap A_3|<t$. For any $B_1,B_2,B_3\in \tilde{\hf}_{2}(T)$, by shiftedness $A_i\cup B_i\in \hf$, $i=1,2,3$. Applying Lemma \ref{lem-2.7} with $(s,r,t)=(t+3,3,t)$  we have
\[
|B_1\cap B_2\cap B_3| \geq 2(t+3)+t-3(t+1) =3.
\]
Thus $\tilde{\hf}_{2}(T)$ is 3-wise 3-intersecting.
\end{proof}

\begin{lemma}\label{lem-3.2}
Assume $k\geq t+3$ and $t\geq1$.
    Suppose that $\hf\subset \binom{[n]}{k}$ is a shifted non-trivial 3-wise $t$-intersecting family. Then
    \begin{itemize}
        \item[(i)]  For $i=4,\ldots,t+3$ and for any $T\in \binom{[t+3]}{t+3-i}$, $\thf_2([t+1]), \thf_i(T)$ are cross $(2i-1)$-intersecting.
        \item[(ii)]  For any $T\in \binom{[t+3]}{t}$, $\thf_{2}([t+1])$, $\thf_3(T)$ are cross 3-intersecting if $T=[t]$ and $\thf_{2}([t+1]), \thf_3(T)$ are cross $5$-intersecting if $T\neq [t]$.
        \item[(iii)] For any $T\in \binom{[t+3]}{t+1}$, $\thf_{2}([t+1])$, $\thf_{2}(T)$ are cross 2-intersecting if $T\cap [t+1]=[t]$ and  $\thf_{2}([t+1]),\thf_{2}(T)$ are cross 3-intersecting if $[t]\not\subset T$.
        \item[(iv)] For any $j\in [t+1]$, $\thf_{2}([t+1])$, $\thf_1([t+3]\setminus \{j\})$ are cross-intersecting.
    \end{itemize}
\end{lemma}

\begin{proof}
Let us first show (i). Let $A_1=[t+1]$ and $A_2=A_3=T$. Then $|A_1\cap A_2\cap A_3|\leq t+3-i<t$. Then for $B_1\in \thf_2([t+1])$ and $B_2\in \thf_i(T)$, by Lemma \ref{lem-2.7}  we have
\[
|B_1\cap B_2|=|B_1\cap B_2\cap B_2| \geq 2(t+3)+t-|A_1|-|A_2|-|A_3| =3t+6-(t+1)-2(t+3-i) = 2i-1.
\]
Thus $\thf_2([t+1]), \thf_i(T)$ are cross $(2i-1)$-intersecting.

Next we show (ii). Let $A_1=[t+1]$ and $A_2=[t]$. Then $|A_1\cap A_2|= t<t+1$. For $B_1\in \thf_2([t+1])$ and $B_2\in \thf_3([t])$, applying Lemma \ref{lem-2.7} with $(s,r,t)=(t+3,2,t+1)$  we have
\[
|B_1\cap B_2| \geq (t+3)+(t+1)-|A_1|-|A_2| =2t+4-(t+1)-t = 3.
\]
Thus $\thf_{2}([t+1])$, $\thf_3([t])$ are cross 3-intersecting.

Let $T\in \binom{[t+3]}{t}$ with $T\neq [t]$ and let $A_1=[t+1]$, $A_2=[t]$, $A_3=[t-1]\cup \{t+1\}$. Then $|A_1\cap A_2\cap A_3| =t-1<t$.  For $B_1\in \thf_2([t+1])$ and $B_2\in \thf_3(T)$, by shiftedness $A_1\cup B_1,A_2\cup B_2, A_3\cup B_2\in \hf$.   Applying Lemma \ref{lem-2.7} with $(s,r,t)=(t+3,3,t)$  we have
\[
|B_1\cap B_2| =|B_1\cap B_2\cap B_2| \geq 2(t+3)+t-|A_1|-|A_2|-|A_3| =3t+6-(t+1)-2t = 5.
\]
Thus $\thf_{2}([t+1])$, $\thf_3(T)$ are cross 5-intersecting.

Thirdly, we show (iii). Let $T\in \binom{[t+3]}{t+1}$. If $T\cap [t+1]=[t]$, then let $A_1=[t+1]$, $A_2=[t]\cup \{t+2\}$. Clearly $|A_1\cap A_2|=t<t+1$. For any $B_1\in \thf_2([t+1])$ and $B_2\in \thf_2(T)\subset \thf_2(A_2)$, by applying Lemma \ref{lem-2.7} with $(s,r,t)=(t+3,2,t+1)$ we have
\[
|B_1\cap B_2| \geq (t+3)+(t+1)-|A_1|-|A_2| = 2.
\]
Thus $\thf_{2}([t+1])$, $\thf_{2}(T)$ are cross 2-intersecting.

If $[t]\not\subset T$, then let $A_1=[t+1]$, $A_2=[t-1]\cup \{t+1,t+2\}$, $A_3=[t]\cup \{t+2\}$. Clearly $|A_1\cap A_2\cap A_3|=t-1<t$. For any $B_1\in \thf_2([t+1])$ and $B_2\in \thf_2(T)\subset \thf_2(A_2)\cap \thf_2(A_3)$, by applying Lemma \ref{lem-2.7} with $(s,r,t)=(t+3,3,t)$ we have
\[
|B_1\cap B_2|=|B_1\cap B_2\cap B_2| \geq 2(t+3)+t-|A_1|-|A_2|-|A_3| =3.
\]
Thus $\thf_{2}([t+1])$, $\thf_{2}(T)$ are cross 3-intersecting.

Finally we show (iv). Since $\hf$ is 2-wise $(t+1)$-intersecting, $\thf_{2}([t+1])$, $\thf_1([t+3]\setminus \{j\})$ are cross-intersecting for $j\in [t+1]$. Thus (iv) holds.
\end{proof}

\section{Proof of Theorem \ref{thm-main}}

\begin{lemma}
Suppose that $0<\alpha<1$, $1\leq\ell<k-t$ and $k-t-1\leq\alpha(n-k-1)$. Then
    \begin{align}\label{ineq-3.0}
    \binom{n-t-3}{k-t-\ell-1} \leq \alpha\binom{n-t-3}{k-t-\ell}.
    \end{align}
Moreover,
 \begin{align}\label{ineq-3.1}
 \binom{n-t-3}{k-t-6}\leq\alpha^6\binom{n-t-3}{k-t}.
 \end{align}
\end{lemma}
\begin{proof}
Note that
\begin{align*}
\frac{\binom{n-t-3}{k-t-\ell}}{\binom{n-t-3}{k-t-\ell-1}} =  \frac{n-k+\ell-2}{k-t-\ell} \geq \frac{n-k-1}{k-t-1} \geq \frac{1}{\alpha}.
\end{align*}
Thus \eqref{ineq-3.0} follows.

For $k-t<6$, $\binom{n-t-3}{k-t-6}=0$ and \eqref{ineq-3.1} holds trivially. For $k-t\geq6$, we have
\[
\frac{\binom{n-t-3}{k-t-6}}{\binom{n-t-3}{k-t}}
=\prod_{\ell=0}^{5}\frac{k-t-\ell}{n-k+\ell-2}.
\]
The product of the first and third factors satisfies
\begin{align*}
\frac{(k-t)(k-t-2)}{(n-k)(n-k-2)}
&=\frac{(k-t-1)^2-1}{(n-k-1)^2-1}\leq\left(\frac{k-t-1}{n-k-1}\right)^2\leq\alpha^2,
\end{align*}
and each of the other four factors is at most $\alpha$ by \eqref{ineq-3.0}. Thus \eqref{ineq-3.1} follows.
\end{proof}

Recall an easy statement about regular bipartite graphs.

\begin{proposition}
Let $\hk$ be a bi-regular non-empty bipartite graph with partite sets $A$, $B$ and $\beta>0$ an arbitrary constant. Then for an arbitrary independent subset $A_0\cup B_0$ with $A_0\subset A$, $B_0\subset B$,
\begin{align}\label{ineq-prop4.2}
|A_0|+\beta|B_0|\leq \max\{|A|,\beta|B|\}.
\end{align}
\end{proposition}

\begin{proposition}
Let $m,s$ be integers with $m>2s\geq 2$, $X$ an $m$-element set, $\hp\subset \binom{X}{s}$, $\hr\subset \binom{X}{s+1}$. Suppose that $\hp$ and $\hr$ are cross-intersecting and $|\hp|\geq \binom{m-1}{s-1}+\binom{m-2}{s-1}$. Then for any $\beta>0$,
\begin{align}\label{ineq-prop4.3}
|\hp|+\beta|\hr|\leq \max\left\{\binom{m}{s},\binom{m-1}{s-1}+\binom{m-2}{s-1}+\beta\binom{m-2}{s-1}\right\}.
\end{align}
\end{proposition}

\begin{proof}
Let $X=\{x_1,x_2,\ldots,x_m\}$. In view of Hilton's Lemma we may assume that $\hp$ and $\hr$ are initial families in the lexicographic order. Hence,
\[
\hp =\left\{P\in \binom{X}{s}\colon P\cap \{x_1,x_2\} \neq \emptyset\right\} \cup \hp_0 \mbox{ with }\hp_0\subset \binom{\{x_3,\ldots,x_m\}}{s}.
\]
By the cross-intersecting property, $\{x_1,x_2\}\subset R$ for all $R\in \hr$ whence $|\hr|=|\hr(\{x_1,x_2\})|$.

Let $\hk$ be the bipartite Kneser graph on $\binom{\{x_3,\ldots,x_m\}}{s}$ and $\binom{\{x_3,\ldots,x_m\}}{s-1}$ with $P,R$ forming an edge iff $P\cap R=\emptyset$. By the cross-intersecting property, $\hp_0\cup \hr(\{x_1,x_2\})$ is an independent set in $\hk$. Now \eqref{ineq-prop4.2} implies \eqref{ineq-prop4.3}.
\end{proof}

We say that $\hf\subset \binom{[n]}{k}$ is {\it exactly 2-wise $t$-intersecting} if $\hf$ is  2-wise $t$-intersecting and there exist $F,F'\in \hf$ such that $|F\cap F'|=t$.

\begin{lemma}\label{lem-4.5}
Assume $k\geq t+3$ and $t\geq1$.
Let  $\hf\subset \binom{[n]}{k}$ be a shifted and non-trivial $3$-wise $t$-intersecting family. If  $\hf$ is exactly 2-wise $(t+1)$-intersecting and $n\geq 3k-2t$, then
\[
|\hf| \leq \max\left\{\binom{n-t}{k-t}, |\ha(n,k,t+3)|\right\}.
\]
\end{lemma}

\begin{proof}
Since $\hf$ is exactly 2-wise $(t+1)$-intersecting, by shiftedness there exist $G,H\in\hf$ with $G\cap H=[t+1]$.
Since $\hf$ is 3-wise $t$-intersecting, $|F\cap[t+1]|\geq t$ for every $F\in\hf$.
Set $\hg_i =\hf([t+1]\setminus \{i\},[t+1])$, $i=1,2,\ldots,t+1$ and set $\hg_0=\hf([t+1],[t+1])$. Then $\hg_1,\hg_2,\ldots,\hg_{t+1}$ are pairwise cross 2-intersecting. By Theorem \ref{thm-lz} and $n-t-1\geq 3(k-t-1)$,
\[
\sum_{1\leq i\leq t+1} |\hg_i| \leq \max\left\{\binom{n-t-1}{k-t}, (t+1)\binom{n-t-3}{k-t-2}\right\}.
\]
If $\binom{n-t-1}{k-t}\geq (t+1)\binom{n-t-3}{k-t-2}$, then
\[
|\hf| =\sum_{1\leq i\leq t+1} |\hg_i| +|\hg_0| \leq \binom{n-t-1}{k-t}+\binom{n-t-1}{k-t-1}=\binom{n-t}{k-t}.
\]
Thus we may assume $\binom{n-t-1}{k-t}< (t+1)\binom{n-t-3}{k-t-2}$.

If $|\hg_0|\leq \binom{n-t-2}{k-t-2}+\binom{n-t-3}{k-t-2}$, then
\[
|\hf| =\sum_{1\leq i\leq t+1} |\hg_i| +|\hg_0| \leq (t+1)\binom{n-t-3}{k-t-2}+\binom{n-t-2}{k-t-2}+\binom{n-t-3}{k-t-2}=|\ha(n,k,t+3)|.
\]
Thus we may assume that $|\hg_0|> \binom{n-t-2}{k-t-2}+\binom{n-t-3}{k-t-2}$.

By shiftedness, $\hg_1\subset \hg_2\subset \ldots\subset \hg_{t+1}$. Since $\hg_0$, $\hg_{t+1}$ are cross-intersecting, by \eqref{ineq-prop4.3}  we infer that
\begin{align*}
|\hf|&\leq |\hg_0|+(t+1)|\hg_{t+1}| \\[3pt]
&\leq  \max\left\{\binom{n-t-1}{k-t-1},\binom{n-t-2}{k-t-2}+\binom{n-t-3}{k-t-2} +(t+1)\binom{n-t-3}{k-t-2}\right\}\\[3pt]
&= \max\left\{\binom{n-t-1}{k-t-1},|\ha(n,k,t+3)|\right\}\\[3pt]
&\leq \max\left\{\binom{n-t}{k-t},|\ha(n,k,t+3)|\right\}.
\end{align*}
\end{proof}

\begin{fact}\label{fact-key1}
Assume $k\geq t+3$ and $t\geq1$.
Let  $\hf\subset \binom{[n]}{k}$ be a shifted, non-trivial $3$-wise $t$-intersecting family. If $\hf\not\subset \mathcal{A}(n,k,t+3)$, then $\thf_1([t+3]\setminus \{i\})$ and $\thf_1([t+3]\setminus \{j\})$  are cross-intersecting for all $1\leq i<j\leq t+1$.
\end{fact}

\begin{proof}
Since $\hf\not\subset \mathcal{A}(n,k,t+3)$, there exists $H\in \hf$ satisfying $|H\cap [t+3]|\leq t+1$. Since $H_0:=[t+1]\cup [t+4,k+2]\prec H$, by shiftedness  $H_0\in \hf$.

If $\thf_1([t+3]\setminus \{i\})$ and $\thf_1([t+3]\setminus \{j\})$  are not cross-intersecting, then we can find $G_i\in \thf_1([t+3]\setminus \{i\})$, $G_j\in \thf_1([t+3]\setminus \{j\})$ with $G_i\cap G_j=\emptyset$. Let  $F_i=G_i\cup [t+3]\setminus \{i\}$, $F_j=G_j\cup [t+3]\setminus \{j\}$. Then $H_0\cap F_i\cap F_j=[t+1]\setminus \{i,j\}$ has size $t-1$, a contradiction.
\end{proof}

\begin{lemma}\label{lem-key1}
Assume $k\geq t+3$ and $t\geq 40$.
    Let  $\hf\subset \binom{[n]}{k}$ be a shifted, non-trivial $3$-wise $t$-intersecting family and let $n\geq k+1+\frac{k-t-1}{\alpha}$ with $\alpha=\sqrt{\frac{\theta}{t+3}}$, $\theta=1.696$.  If $|\hf|>\max\{\binom{n-t}{k-t},|\ha(n,k,t+3)|\}$ then
\begin{align}\label{ineq-4.4}
|\thf_{2}([t+1])|> \binom{n-t-3}{k-t-3}.
\end{align}
\end{lemma}

\begin{proof}
By assumption, $k-t-1\leq\alpha(n-k-1)$. Since $t\geq 4$, we infer that
\[
\alpha=\sqrt{\frac{1.696}{t+3}}
\le\sqrt{\frac{1.696}{7}}<\frac12.
\]
It follows that
\[
n-k\ge1+\frac{k-t-1}{\alpha}
>1+2(k-t-1)=2(k-t)-1.
\]
Thus $n-k\geq2(k-t)$. By Lemma \ref{lem-4.5}, we may assume that $\hf$ is 2-wise $(t+2)$-intersecting and we distinguish two  cases.

{\bf Case 1.} $\hf$ is exactly  2-wise $(t+2)$-intersecting.

Then by shiftedness there exist $G,H\in \hf$ with $G\cap H=[t+2]$ whence
\begin{align}\label{ineq-4.2}
|F\cap [t+2]| \geq t \mbox{ for all }F\in \hf.
\end{align}
It follows that $\hf_i=\emptyset$ for all $i\geq 4$.

{\bf Case 1.1. } There exist $F_1,F_2\in \hf$ with $|F_1\cap F_2|=t+2$ and $|F_1\cap F_2\cap [t+2]|\leq t$.

By shiftedness, we may assume that  $F_1\cap F_2=[t]\cup \{t+3,t+4\}$. Then for any $F\in \hf$,
\begin{align}\label{ineq-4.1}
|F\cap [t+2]| \geq t,\ |F\cap ([t]\cup \{t+3,t+4\})| \geq t \mbox{ for all }F\in \hf.
\end{align}
Thus for any $F\in \hf$, one of (i), (ii), (iii) holds:
\begin{itemize}
  \item[(i)] $[t]\subset F$.
  \item[(ii)] $|F\cap [t]|=t-1$ and $|F\cap [t+1,t+4]|\geq 3$.
  \item[(iii)] $|F\cap [t]|=t-2$ and $[t+1,t+4]\subset F$.
\end{itemize}
To verify (ii), suppose that $|F\cap[t]|=t-1$ and $|F\cap[t+1,t+4]|=2$. By \eqref{ineq-4.1}, $F$ contains one element of $\{t+1,t+2\}$ and one element of $\{t+3,t+4\}$. Then $F_0:=[t-1]\cup \{t+1,t+2\}\cup(F\setminus [t+4])\prec F$. By shiftedness, $F_0\in \hf$. However $F_0\cap ([t]\cup\{t+3,t+4\})=[t-1]$, contradicting \eqref{ineq-4.1}. Since $\hf$ is 2-wise $(t+2)$-intersecting, by Lemma \ref{lem-3.1} (ii), $\thf_3([t])$ is 2-wise 5-intersecting. For $F\in \hf_3$, cases (ii) and (iii) would require $|F\cap [t+4]|\geq t+2$. However $F\in \hf_3$ implies that $T:=F\cap [t+3]$ is a $t$-set. Thus only $T=[t]$ is possible.  In view of Lemma \ref{lem-F20},
\[
|\hf_3|= |\thf_3([t])|  \leq \binom{n-t-3}{k-t-5} \leq\binom{n-t-3}{k-t-4}.
\]
For $F\in \hf_2$, the possible sets $T=F\cap[t+3]$, other than $[t+1]$, consist of two sets containing $[t]$, at most $3t$ sets with $|T\cap[t]|=t-1$, and at most $\binom t2$ sets with $|T\cap[t]|=t-2$. The last two counts follow by choosing the missing elements of $[t]$ and the required elements of $\{t+1,t+2,t+3\}$. Each $\thf_2(T)$ is 2-wise 3-intersecting by Lemma \ref{lem-3.1} (iii) with $j=2$.
Thus by Lemma \ref{lem-F20},
\[
|\hf_2|-|\thf_{2}([t+1])| \leq \left(2+3t+\binom{t}{2}\right)\binom{n-t-3}{k-t-4}<\binom{t+3}{2}\binom{n-t-3}{k-t-4}.
\]
Since $\hf$ is 2-wise $(t+2)$-intersecting, we infer that $\thf_1([t+3]\setminus \{i\})$ and $\thf_1([t+3]\setminus \{j\})$ are cross-intersecting for $1\leq i<j\leq t+3$.
By Theorem \ref{thm-hilton},
\begin{align*}
|\hf_1|&\leq\max\left\{\binom{n-t-3}{k-t-2},
                     (t+3)\binom{n-t-4}{k-t-3}\right\}\leq\frac{(t+3)\alpha}{1+\alpha}\binom{n-t-3}{k-t-2},
\end{align*}
which follows from
\[
\frac{\binom{n-t-4}{k-t-3}}{\binom{n-t-3}{k-t-2}}
=\frac{k-t-2}{n-t-3}<\frac{\alpha}{1+\alpha}\mbox{ and } \frac{(t+3)\alpha}{1+\alpha}>1.
\]
Clearly $|\hf_0|\leq \binom{n-t-3}{k-t-3}$. Thus,
\[
\begin{aligned}
|\hf|-|\thf_{2}([t+1])|
\leq\binom{t+3}{2}\binom{n-t-3}{k-t-4}
     +\frac{(t+3)\alpha}{1+\alpha}\binom{n-t-3}{k-t-2}+\binom{n-t-3}{k-t-3}.
\end{aligned}
\]
By $n\geq k+1+\frac{k-t-1}{\alpha}$ and \eqref{ineq-3.0},
 \begin{align*}
 &\quad \binom{t+3}{2}\binom{n-t-3}{k-t-4}+ \frac{(t+3)\alpha}{1+\alpha}\binom{n-t-3}{k-t-2} \\[3pt]
 &\leq \frac{\theta}{2}(t+2)\binom{n-t-3}{k-t-2}+ \frac{(t+3)\alpha}{1+\alpha}\binom{n-t-3}{k-t-2}\\[3pt]
 &= \left(\frac{\theta}{2}(t+2)+ \frac{(t+3)\alpha}{1+\alpha}\right)\binom{n-t-3}{k-t-2}.
 \end{align*}
If $|\thf_{2}([t+1])|\leq \binom{n-t-3}{k-t-3}<\sqrt{\frac{\theta}{t+3}}\binom{n-t-3}{k-t-2}$, then for $t\geq40$,
\begin{align*}
|\hf| &\leq \left(\frac{\theta}{2}(t+2)+ \frac{(t+3)\alpha}{1+\alpha}+\sqrt{\frac{\theta}{t+3}}\right)\binom{n-t-3}{k-t-2}+\binom{n-t-3}{k-t-3}\\[3pt]
&\leq (t+3)\binom{n-t-3}{k-t-2} +\binom{n-t-3}{k-t-3}=|\mathcal{A}(n,k,t+3)|,
\end{align*}
a contradiction.

{\bf Case 1.2. } $|F_1\cap F_2\cap [t+2]|\geq t+1$ for any $F_1, F_2\in \hf$ with $|F_1\cap F_2|=t+2$.

Let $1\leq i<j\leq t+2$.
For any $ E_i\in \thf_1([t+3]\setminus \{i\})$ and $E_j\in \thf_1([t+3]\setminus \{j\})$, let $F_i=E_i\cup[t+3]\setminus \{i\}$ and $F_j=E_j\cup[t+3]\setminus \{j\}$, clearly we have  $|F_i\cap F_j\cap [t+2]|= t$. Then our assumption implies that $|F_i\cap F_j|=t+1 +|E_i\cap E_j|\geq t+3$. It follows that $|E_i\cap E_j|\geq 2$. Thus  $\thf_1([t+3]\setminus \{i\})$ and $\thf_1([t+3]\setminus \{j\})$ are cross 2-intersecting for $1\leq i<j\leq t+2$.
By  Theorem \ref{thm-lz}, we have
\begin{align*}
|\hf_1|=&|\thf_1([t+2])|+\sum_{1\leq i\leq t+2} |\thf_1([t+3]\setminus \{i\})| \\[3pt]
&\leq \binom{n-t-3}{k-t-2}+\max\left\{\binom{n-t-3}{k-t-2},(t+2)\binom{n-t-5}{k-t-4}\right\}\\[3pt]
&<(\theta+1)\binom{n-t-3}{k-t-2}.
\end{align*}

By \eqref{ineq-4.2} we infer that $\thf_3(T)=\emptyset$ for any $T\in \binom{[t+3]}{t}\setminus \binom{[t+2]}{t}$. Thus by Lemma \ref{lem-3.1} (ii) and Lemma \ref{lem-F20},
\[
|\hf_3| =\sum_{T\in \binom{[t+2]}{t}} |\thf_3(T)| \leq \left(\binom{t+2}{2}-1\right) \binom{n-t-3}{k-t-6} +\binom{n-t-3}{k-t-4}.
\]
By Lemma \ref{lem-3.1} (iii) and Lemma \ref{lem-F20},
\[
|\hf_2|-|\thf_{2}([t+1])|=\sum_{T\in \binom{[t+3]}{t+1}\setminus \{[t+1]\}, |T\cap [t+2]|\geq t}|\thf_2(T)| \leq \left(\binom{t+2}{2}+(t+1)\right)\binom{n-t-3}{k-t-4}.
\]
Clearly $|\hf_0|\leq \binom{n-t-3}{k-t-3}$. Thus,
\[
|\hf|-|\thf_{2}([t+1])|\leq \binom{t+2}{2} \binom{n-t-3}{k-t-6}+ \binom{t+3}{2}\binom{n-t-3}{k-t-4}+(\theta+1)\binom{n-t-3}{k-t-2}+\binom{n-t-3}{k-t-3}.
\]
By $n\geq k+1+\frac{k-t-1}{\alpha}$ and \eqref{ineq-3.0},
 \begin{align*}
 &\quad\binom{t+2}{2} \binom{n-t-3}{k-t-6}+ \binom{t+3}{2}\binom{n-t-3}{k-t-4} \\[3pt]
 &\leq \frac{\theta^2}{2}\binom{n-t-3}{k-t-2}+ \frac{\theta}{2}(t+2)\binom{n-t-3}{k-t-2}\\[3pt]
 &= \left(\frac{\theta^2}{2}+\frac{\theta}{2}(t+2)\right)\binom{n-t-3}{k-t-2}.
 \end{align*}
If $|\thf_{2}([t+1])|\leq \binom{n-t-3}{k-t-3}<\sqrt{\frac{\theta}{t+3}}\binom{n-t-3}{k-t-2}$, then for $t\geq40$,
\begin{align*}
|\hf| &\leq \left(\frac{\theta^2}{2}+\frac{\theta}{2}(t+2)
+\theta+1+\sqrt{\frac{\theta}{t+3}}\right)\binom{n-t-3}{k-t-2}
+\binom{n-t-3}{k-t-3}\\[3pt]
&\leq (t+3)\binom{n-t-3}{k-t-2} +\binom{n-t-3}{k-t-3}=|\mathcal{A}(n,k,t+3)|,
\end{align*}
a contradiction.  Thus $|\thf_{2}([t+1])|> \binom{n-t-3}{k-t-3}$.

{\bf Case 2.} $\hf$ is 2-wise $(t+3)$-intersecting.

By Lemma \ref{lem-3.1} (i) and Lemma \ref{lem-F20},
\[
\sum_{4\leq i\leq  t+3} |\hf_i| \leq \sum_{4\leq i\leq t+3} \binom{t+3}{i}\binom{n-t-3}{k-t-2i}.
\]
By Lemma \ref{lem-3.1} (ii) and Lemma \ref{lem-F20},
\[
    |\hf_3| \leq \binom{t+3}{3}\binom{n-t-3}{k-t-6}.
\]
By Lemma \ref{lem-3.1} (iii) and Lemma \ref{lem-F20},
\[
|\hf_2|-|\thf_2([t+1])|=\sum_{T\in \binom{[t+3]}{t+1}\setminus \{[t+1]\}}|\thf_2(T)| \leq \left(\binom{t+3}{2}-1\right)\binom{n-t-3}{k-t-5}.
\]
 Since $\hf$ is 2-wise $(t+3)$-intersecting, we infer that $\thf_1([t+3]\setminus \{i\}), \thf_1([t+3]\setminus \{j\})$ are cross 2-intersecting. By Theorem \ref{thm-lz}, since $n-t-3\geq 3(k-t-3)$
\[
|\hf_1| \leq\max\left\{\binom{n-t-3}{k-t-2},(t+3)\binom{n-t-5}{k-t-4}\right\}< \theta\binom{n-t-3}{k-t-2}.
\]
Clearly $|\hf_0|\leq \binom{n-t-3}{k-t-3}$.
Thus,
\begin{align*}
|\hf| -|\thf_{2}([t+1])|&\leq \sum_{3\leq i\leq t+3} \binom{t+3}{i}\binom{n-t-3}{k-t-2i} +\binom{t+3}{2}\binom{n-t-3}{k-t-5}+\theta\binom{n-t-3}{k-t-2}\\[3pt]
&\qquad\qquad +\binom{n-t-3}{k-t-3}.
\end{align*}
Note $\alpha^2= \frac{\theta}{t+3}$. By $n\geq k+1+\frac{k-t-1}{\alpha}$ and \eqref{ineq-3.0}, for $i\geq 4$ we have
\begin{align*}
 \binom{t+3}{i}\binom{n-t-3}{k-t-2i} \leq   \binom{t+3}{i-1} \frac{t+4-i}{i} \alpha^2 \binom{n-t-3}{k-t-2i+2}\leq \frac{\theta}{4}\binom{t+3}{i-1}\binom{n-t-3}{k-t-2i+2}.
\end{align*}
 Thus,
\begin{align}\label{ineq-4.3}
 \sum_{3\leq i\leq t+3} \binom{t+3}{i}\binom{n-t-3}{k-t-2i} &\leq \frac{4}{4-\theta}\binom{t+3}{3} \binom{n-t-3}{k-t-6}\nonumber\\[3pt]
 &\leq \frac{4}{4-\theta} \frac{\theta^2(t+3)(t+2)(t+1)}{6(t+3)^2}\binom{n-t-3}{k-t-2}\nonumber\\[3pt]
 & \leq \frac{2\theta^2(t+1)(t+2)}{3(4-\theta)(t+3)}\binom{n-t-3}{k-t-2}.
\end{align}
Moreover,
\[
\binom{t+3}{2}\binom{n-t-3}{k-t-5} \leq \frac{\theta}{2}(t+3)\binom{n-t-3}{k-t-3}<\frac{\theta^{3/2}}{2}\sqrt{t+3}\binom{n-t-3}{k-t-2}.
\]
Therefore,
\[
|\hf| -|\thf_{2}([t+1])| <\left( \frac{2\theta^2(t+1)(t+2)}{3(4-\theta)(t+3)}+\frac{\theta^{3/2}}{2}\sqrt{t+3}+\theta\right)\binom{n-t-3}{k-t-2}+\binom{n-t-3}{k-t-3}.
\]
If $|\thf_{2}([t+1])|\leq \binom{n-t-3}{k-t-3}<\sqrt{\frac{\theta}{t+3}}\binom{n-t-3}{k-t-2}$, then for $t\geq40$,
\begin{align*}
|\hf| &\leq \left( \frac{2\theta^2(t+1)(t+2)}{3(4-\theta)(t+3)}
+\frac{\theta^{3/2}}{2}\sqrt{t+3}+\theta+\sqrt{\frac{\theta}{t+3}}\right)\binom{n-t-3}{k-t-2}+\binom{n-t-3}{k-t-3}\\[3pt]
&<(t+3)\binom{n-t-3}{k-t-2} +\binom{n-t-3}{k-t-3}=|\mathcal{A}(n,k,t+3)|,
\end{align*}
a contradiction.  Thus $|\thf_{2}([t+1])|> \binom{n-t-3}{k-t-3}$.
\end{proof}

\begin{proof}[Proof of Theorem \ref{thm-main}]
For $k=t+2$, suppose that $[t+2], [t+1]\cup\{t+3\}\in \hf$, then the non-triviality implies $F\subset [t+3]$ for any $F\in \hf$. Thus $\hf\subset \ha(n,k,t+3)$ and we are done. Thus we may assume $k\geq t+3$.

Since a trivial family has size at most $\binom{n-t}{k-t}$, we may assume that $\hf\subset \binom{[n]}{k}$ is a non-trivial $3$-wise $t$-intersecting family of maximum size. Since $n\geq2k-t$, Corollary \ref{cor-2.1} permits us to assume that $\hf$ is shifted. To obtain a contradiction, assume $|\hf|>\max\{\binom{n-t}{k-t},|\ha(n,k,t+3)|\}$.

If $(t+4,t+6,\ldots, 2k-t)\notin \thf_{2}([t+1])$, then applying Proposition \ref{prop-F78} (i) with $p=1$,
\begin{align*}
|\thf_{2}([t+1])| \leq \binom{n-t-3}{k-t-3},
\end{align*}
contradicting \eqref{ineq-4.4}. Thus we may assume $(t+4,t+6,\ldots, 2k-t)\in \thf_{2}([t+1])$.

By Lemma \ref{lem-3.2} (iv),  $\thf_1([t+3]\setminus \{j\}),\thf_{2}([t+1])$ are cross-intersecting for each $j\in [t+1]$. It follows that
\[
(t+5, t+7,\ldots, 2k-t-1)\notin \thf_1([t+3]\setminus \{j\}).
\]
By applying Proposition \ref{prop-F78} (i) with $p=0$,
\begin{align*}
|\thf_1([t+3]\setminus \{j\})| \leq \binom{n-t-3}{k-t-3} \mbox{ for each }j=1,2,\ldots,t+1.
\end{align*}
Thus,
\begin{align}\label{ineq-new3.1}
|\hf_1| \leq (t+1)\binom{n-t-3}{k-t-3}+2\binom{n-t-3}{k-t-2}.
\end{align}

By Lemma \ref{lem-3.2} (i), $\thf_{2}([t+1]), \thf_i(T)$ are cross $(2i-1)$-intersecting for $i=4,\ldots,t+3$ and for any $T\in \binom{[t+3]}{t+3-i}$. Since $(t+4,t+6,\ldots, 2k-t)\in \thf_{2}([t+1])$, we infer that  for any $T\in \binom{[t+3]}{t+3-i}$,
\[
[t+4,t+4i-2]\cup (t+4i-1,t+4i+1,\ldots ) \notin \thf_i(T).
\]
Applying Proposition \ref{prop-F78} (i) with $p=4i-4$,
\[
|\thf_i(T)| \leq \binom{n-t-3}{k-(t+3-i)-(4i-4)-1} = \binom{n-t-3}{k-t-3i}.
\]
Note that if $k-t-3i\leq -1$ then
\[
|[t+4,k+i]\cap (t+4,t+6,\ldots,2k-t)|=\min\left\{ k-t-1,\left\lfloor \frac{k-t+i}{2}\right\rfloor-1\right\}\leq 2i-2.
\]
By shiftedness, we infer that $\thf_i(T)=\emptyset$. Thus the upper bound $\binom{n-t-3}{k-t-3i}$ remains valid when $k-t-3i$ is negative.
Therefore,
\begin{align}\label{ineq-new3.2}
\sum_{4\leq i\leq t+3} |\hf_i| \leq \sum_{4\leq i\leq t+3} \binom{t+3}{i} \binom{n-t-3}{k-t-3i}.
\end{align}

By Lemma \ref{lem-3.2} (ii) for any $T\in \binom{[t+3]}{t}$, $\thf_{2}([t+1])$, $\thf_3(T)$ are cross 3-intersecting if $T=[t]$ and $\thf_{2}([t+1]), \thf_3(T)$ are cross $5$-intersecting if $T\neq [t]$. Then $(t+4,t+6,\ldots, 2k-t)\in \thf_{2}([t+1])$ implies
\begin{align*}
    &[t+4,t+6]\cup (t+7,t+9,\ldots ) \notin \thf_3([t]),\\[3pt]
    &[t+4,t+10]\cup (t+11,t+13,\ldots ) \notin \thf_3(T), \mbox{ for }T\neq [t].
\end{align*}
Applying Proposition \ref{prop-F78} (i) with $p=4$ and $p=8$, respectively,
\begin{align*}
    |\thf_3(T)|\leq  \left\{\begin{array}{ll}
                 \binom{n-t-3}{k-t-5}, & \mbox{ if } T=[t]; \\[5pt]
                 \binom{n-t-3}{k-t-9}, & \mbox{ if } T\neq [t].
                \end{array}
\right.
\end{align*}
Thus,
\begin{align}\label{ineq-new3.3-1}
|\hf_3| \leq  \left(\binom{t+3}{3}-1\right)\binom{n-t-3}{k-t-9}+ \binom{n-t-3}{k-t-5}.
\end{align}

By Lemma \ref{lem-3.2} (iii), for any $T\in \binom{[t+3]}{t+1}$, $\thf_{2}([t+1])$, $\thf_{2}(T)$ are cross 2-intersecting if $T\cap [t+1]=[t]$ and  $\thf_{2}([t+1]),\thf_{2}(T)$ are cross 3-intersecting if $[t]\not\subset T$. Then
\begin{align*}
    &\{t+4\}\cup (t+5,t+7,\ldots ) \notin \thf_2(T), \mbox{ for }T\cap [t+1]=[t], \\[3pt]
    &\{t+4,t+5,t+6\}\cup (t+7,t+9,\ldots ) \notin \thf_2(T), \mbox{ for }[t]\not\subset T.
\end{align*}
Applying Proposition \ref{prop-F78} (i) with $p=2$ and $p=4$, respectively,
\begin{align*}
    |\thf_2(T)|\leq  \left\{\begin{array}{ll}
                 \binom{n-t-3}{k-t-4}, & \mbox{ if } T\cap [t+1]=[t];\\[5pt]
                 \binom{n-t-3}{k-t-6}, & \mbox{ if } [t]\not\subset T.
                \end{array}
                \right.
\end{align*}
Obviously $|\thf_2([t+1])|\leq \binom{n-t-3}{k-t-1}$. Thus,
\begin{align}\label{ineq-new3.3-2}
|\hf_2|\leq 2\binom{n-t-3}{k-t-4}&+\left(\binom{t+3}{2}-3\right)\binom{n-t-3}{k-t-6}+\binom{n-t-3}{k-t-1}.
\end{align}
Clearly $|\hf_0|\leq \binom{n-t-3}{k-t-3}$.
Together with \eqref{ineq-new3.1}, \eqref{ineq-new3.2}, \eqref{ineq-new3.3-1} and \eqref{ineq-new3.3-2},
\begin{align*}
|\hf| &\leq (t+2)\binom{n-t-3}{k-t-3}+2\binom{n-t-3}{k-t-2}+\sum_{2\leq i\leq t+3} \binom{t+3}{i} \binom{n-t-3}{k-t-3i}\\[3pt]
&\qquad + \binom{n-t-3}{k-t-5} +2\binom{n-t-3}{k-t-4}+\binom{n-t-3}{k-t-1}.
\end{align*}
Since
\[
\binom{n-t}{k-t} =\binom{n-t-3}{k-t}+3\binom{n-t-3}{k-t-1}+3\binom{n-t-3}{k-t-2}+\binom{n-t-3}{k-t-3},
\]
to show $|\hf|\leq \binom{n-t}{k-t}$ it suffices to show that
\begin{align}\label{ineq-final}
&(t+1)\binom{n-t-3}{k-t-3}+\sum_{2\leq i\leq t+3} \binom{t+3}{i} \binom{n-t-3}{k-t-3i} +2\binom{n-t-3}{k-t-4} + \binom{n-t-3}{k-t-5}\nonumber\\[3pt]
&\qquad <\binom{n-t-3}{k-t}+2\binom{n-t-3}{k-t-1}+\binom{n-t-3}{k-t-2}.
\end{align}

Note that $\frac{\theta^{3/2}}{\sqrt{t+3}} <1$ for $t\geq 2$. Since for $i\geq 3$,
 \begin{align*}
\binom{t+3}{i} \binom{n-t-3}{k-t-3i} &\leq \binom{t+3}{i-1} \frac{t+4-i}{i} \alpha^3\binom{n-t-3}{k-t-3(i-1)}\\[3pt]
&\leq  \frac{\theta^{3/2}}{3\sqrt{t+3}} \binom{t+3}{i-1}\binom{n-t-3}{k-t-3(i-1)}\\[3pt]
&<\frac{1}{3}\binom{t+3}{i-1}\binom{n-t-3}{k-t-3(i-1)},
 \end{align*}
 by \eqref{ineq-3.0} we infer that
\begin{align*}
\sum_{2\leq i\leq t+3} \binom{t+3}{i} \binom{n-t-3}{k-t-3i} &\leq \frac{3}{2}\binom{t+3}{2} \binom{n-t-3}{k-t-6} \\[3pt]
&\overset{\eqref{ineq-3.1}}{\leq} \frac{3(t+3)(t+2)}{4} \frac{\theta^3}{(t+3)^3}\binom{n-t-3}{k-t}\\[3pt]
&\leq \frac{3\theta^{3}}{4(t+3)}\binom{n-t-3}{k-t}<\binom{n-t-3}{k-t}.
\end{align*}
Moreover, for $t\geq 2$
\begin{align*}
&\quad (t+1)\binom{n-t-3}{k-t-3} +2\binom{n-t-3}{k-t-4} + \binom{n-t-3}{k-t-5}\\[3pt]
&\leq \theta\binom{n-t-3}{k-t-1}+\left(\frac{2\theta}{t+3} + \frac{\theta^{3/2}}{(t+3)^{3/2}}\right) \binom{n-t-3}{k-t-2}\\[3pt]
&<2\binom{n-t-3}{k-t-1}+\binom{n-t-3}{k-t-2}.
\end{align*}
Thus \eqref{ineq-final} follows.
\end{proof}

\begin{proof}[Proof of Theorem \ref{thm-main0}]
Note that
\begin{align*}
&(4t+9)(k-t)^2-(4t+6)(k-t)+1-\left[\sqrt{4t+9}\,(k-t-1)+2\right]^2\\
&=(\sqrt{4t+9}-1)(\sqrt{4t+9}-3)(k-t-1)>0.
\end{align*}
It  follows that
\begin{align*}
n_0(k,t)
&=t+\frac{\sqrt{(4t+9)(k-t)^2-(4t+6)(k-t)+1}-(k-t)+3}{2}\\
&>k+1+\frac{\sqrt{4t+9}-3}{2}(k-t-1).
\end{align*}
Moreover, for $t\geq 42$
\[
\frac{4(t+3)}{(\sqrt{4t+9}-3)^2}<1.696,
\]
since the left hand side  is a
decreasing function of $t$, and the inequality at $t=42$
follows from $\sqrt{177}<1411/106$.
Thus $n_0(k,t)>k+1+\sqrt{(t+3)/1.696}\,(k-t-1)$.
If $n\geq n_0(k,t)$, then Theorem \ref{thm-main} and \eqref{ineq-1.0}
imply $g(n,k,3,t)=\binom{n-t}{k-t}$.
If $k<n<n_0(k,t)$, then \eqref{ineq-1.0} implies
$|\ha(n,k,t+3)|>\binom{n-t}{k-t}$. Since $\ha(n,k,t+3)$ is
3-wise $t$-intersecting, we conclude that $g(n,k,3,t)>\binom{n-t}{k-t}$ for $k<n<n_0(k,t)$.
\end{proof}

\section{Proof of Theorem \ref{thm-main2}}

\begin{lemma}\label{lem-key}
Assume $k\geq t+3$ and $t\geq 45$.
    Let  $\hf\subset \binom{[n]}{k}$ be a shifted and non-trivial $3$-wise $t$-intersecting family and let $n\geq \frac{1}{\alpha}(k-t)+k$ with $\alpha=\sqrt{\frac{\theta}{t+3}}$, $\theta=1.663$.  If $|\hf|>\max\{|\ha(n,k,t+3)|,|\hb(n,k,t+3)|\}$ then
\begin{align}\label{ineq-3.2}
|\thf_{2}([t+1])|> 2\binom{n-t-3}{k-t-2}.
\end{align}
\end{lemma}

\begin{proof}
 By Theorem \ref{thm-1.5}, we may assume $n\leq \frac{t+4}{2}(k-t)$. Note that $k\geq t+3$ and $t\geq45$ imply
\[
\frac{n-t-3}{k-t-2}
\leq\frac{t+4}{2}+\frac{1}{k-t-2}
\leq\frac{t+6}{2}<t+1.
\]

 Since $\hf\not\subset \ha(n,k,t+3)$,  by Fact \ref{fact-key1} and Theorem \ref{thm-hilton} we have
\begin{align}\label{ineq-new5.01}
\sum_{1\leq i\leq t+1} |\thf_1([t+3]\setminus \{i\})| \leq (t+1)\binom{n-t-4}{k-t-3}.
\end{align}

Since $\alpha^2\leq\theta/48<1/25$, we have $n\geq6(k-t)+t$.
In particular, $n-t-3\geq5(k-t-3)$ and $n-t-3\geq5(k-t-4)$,
so Theorem \ref{thm-EKR} applies to the 2-wise 4-intersecting and
2-wise 3-intersecting families used below. If the required intersection
size exceeds the size of a member, the corresponding family is empty.
For $\ell=1,2,3$ and $k-t-\ell-2\geq0$, by $k-t\leq\alpha(n-k)$ we also have
\[
\frac{\binom{n-t-3-\ell}{k-t-2-\ell}}{\binom{n-t-3}{k-t-2}}
=\prod_{i=0}^{\ell-1}\frac{k-t-2-i}{n-t-3-i}
\leq\left(\frac{\alpha}{1+\alpha}\right)^\ell.
\]

Now we distinguish three  cases.

{\bf Case 1.} $\hf$ is exactly 2-wise $(t+1)$-intersecting.

Then by shiftedness there exist $G,H\in \hf$ with $G\cap H=[t+1]$ whence
\[
|F\cap [t+1]| \geq t \mbox{ for all }F\in \hf.
\]
It follows that $\hf_i=\emptyset$ for all $i\geq 4$.
By Lemma \ref{lem-3.1} (ii) and Lemma \ref{lem-F20},
\[
    |\hf_3| =\sum_{i\in [t+1]} |\thf_3([t+1]\setminus \{i\})| \leq t \binom{n-t-3}{k-t-6} +\binom{n-t-3}{k-t-4}.
\]
By Lemma \ref{lem-3.1} (iii) and Lemma \ref{lem-F20},
\[
|\hf_2|-|\thf_2([t+1])|=\sum_{T\in \binom{[t+3]}{t+1}\setminus \{[t+1]\}, |T\cap [t+1]|=t}|\thf_2(T)| \leq 2\binom{n-t-3}{k-t-3}+2t\binom{n-t-3}{k-t-4}.
\]
By \eqref{ineq-new5.01},
\[
|\hf_1|=\sum_{1\leq i\leq t+3} |\thf_1([t+3]\setminus \{i\})| \leq (t+1)\binom{n-t-4}{k-t-3} +2\binom{n-t-3}{k-t-2}.
\]
Clearly $|\hf_0|\leq \binom{n-t-3}{k-t-3}$.
Thus,
\[
|\hf|- |\thf_{2}([t+1])|\leq t \binom{n-t-3}{k-t-6}+ (2t+1)\binom{n-t-3}{k-t-4}+(t+4)\binom{n-t-3}{k-t-3}+2\binom{n-t-3}{k-t-2}.
\]
Note that $\alpha^2\leq\theta/48<1/25$ for $t\geq45$. By $n\geq \frac{1}{\alpha}(k-t)+k$ and \eqref{ineq-3.0},
 \begin{align*}
 &\quad t \binom{n-t-3}{k-t-6}+ (2t+1)\binom{n-t-3}{k-t-4}+(t+3)\binom{n-t-3}{k-t-3}+2\binom{n-t-3}{k-t-2}\\[3pt]
 &\leq (2t+3)\binom{n-t-3}{k-t-4}+(t+3)\binom{n-t-3}{k-t-3}+2\binom{n-t-3}{k-t-2}\\[3pt]
 &\leq 4\binom{n-t-3}{k-t-2}+\frac{t+3}{5}\binom{n-t-3}{k-t-2}+2\binom{n-t-3}{k-t-2}\\[3pt]
 &\leq \left(\frac{t+3}{5}+6\right)\binom{n-t-3}{k-t-2}.
 \end{align*}
If $|\thf_{2}([t+1])|\leq 2\binom{n-t-3}{k-t-2}$, then for $t\geq 10$,
\[
|\hf| \leq \left(\frac{t+3}{5}+8\right)\binom{n-t-3}{k-t-2} +\binom{n-t-3}{k-t-3}\leq (t+3)\binom{n-t-3}{k-t-2} +\binom{n-t-3}{k-t-3}=|\mathcal{A}(n,k,t+3)|,
\]
a contradiction. Thus $|\thf_{2}([t+1])|> 2\binom{n-t-3}{k-t-2}$.

{\bf Case 2.} $\hf$ is exactly  2-wise $(t+2)$-intersecting.

Then by shiftedness there exist $G,H\in \hf$ with $G\cap H=[t+2]$ whence
\begin{align*}
|F\cap [t+2]| \geq t \mbox{ for all }F\in \hf.
\end{align*}
It follows that $\hf_i=\emptyset$ for all $i\geq 4$.
By Lemma \ref{lem-3.1} (ii), Lemma \ref{lem-F20} and Theorem \ref{thm-EKR},
\[
|\hf_3| =\sum_{T\in \binom{[t+2]}{t}} |\thf_3(T)| \leq \left(\binom{t+2}{2}-1\right) \binom{n-t-3}{k-t-6} +\binom{n-t-7}{k-t-4}.
\]
By Lemma \ref{lem-3.1} (iii) and Theorem \ref{thm-EKR},
\[
|\hf_2|-|\thf_{2}([t+1])|=\sum_{T\in \binom{[t+3]}{t+1}\setminus \{[t+1]\}, |T\cap [t+2]|\geq t}|\thf_2(T)| \leq \left(\binom{t+2}{2}+(t+1)\right)\binom{n-t-6}{k-t-4}.
\]
Since $\hf$ is 2-wise $(t+2)$-intersecting, $\thf_1([t+3]\setminus \{i\})$ and $\thf_1([t+3]\setminus \{j\})$ are cross-intersecting for $1\leq i<j\leq t+3$. Since $n\leq \frac{t+4}{2}(k-t)$ implies $\frac{n-t-3}{k-t-2} < t+3$, by  Theorem \ref{thm-hilton} we have
\[
|\hf_1|=\sum_{1\leq i\leq t+3} |\thf_1([t+3]\setminus \{i\})| \leq (t+3)\binom{n-t-4}{k-t-3}.
\]
Clearly $|\hf_0|\leq \binom{n-t-3}{k-t-3}$. Thus,
\[
\begin{aligned}
|\hf|-|\thf_{2}([t+1])|\leq{}&\binom{t+2}{2}\binom{n-t-3}{k-t-6}
+\binom{t+3}{2}\binom{n-t-5}{k-t-4}\\
&{}+(t+3)\binom{n-t-4}{k-t-3}+\binom{n-t-3}{k-t-3}.
\end{aligned}
\]
By $n\geq \frac{1}{\alpha}(k-t)+k$ and \eqref{ineq-3.0},
 \begin{align*}
 &\quad\binom{t+2}{2} \binom{n-t-3}{k-t-6}+ \binom{t+3}{2}\binom{n-t-5}{k-t-4}+(t+3)\binom{n-t-4}{k-t-3} \\[3pt]
 &\leq \frac{\theta^2}{2}\binom{n-t-3}{k-t-2}+ \frac{\theta(t+2)}{2(1+\alpha)^2}\binom{n-t-3}{k-t-2}
 +\frac{(t+3)\alpha}{1+\alpha}\binom{n-t-3}{k-t-2}\\[3pt]
 &= \left(\frac{\theta^2}{2}+\frac{\theta(t+2)}{2(1+\alpha)^2}+
 \frac{(t+3)\alpha}{1+\alpha}\right)\binom{n-t-3}{k-t-2}.
 \end{align*}
If $|\thf_{2}([t+1])|\leq 2\binom{n-t-3}{k-t-2}$, then for $t\geq45$,
\begin{align*}
|\hf| &\leq \left(\frac{\theta^2}{2}+
\frac{\theta(t+2)}{2(1+\alpha)^2}+
\frac{(t+3)\alpha}{1+\alpha}+2\right)\binom{n-t-3}{k-t-2}+\binom{n-t-3}{k-t-3}\\[3pt]
&\leq (t+3)\binom{n-t-3}{k-t-2} +\binom{n-t-3}{k-t-3}=|\mathcal{A}(n,k,t+3)|,
\end{align*}
a contradiction.  Thus $|\thf_{2}([t+1])|> 2\binom{n-t-3}{k-t-2}$.

{\bf Case 3.} $\hf$ is 2-wise $(t+3)$-intersecting.

By Lemma \ref{lem-3.1} (i) and Lemma \ref{lem-F20},
\[
\sum_{4\leq i\leq  t+3} |\hf_i| \leq \sum_{4\leq i\leq t+3} \binom{t+3}{i}\binom{n-t-3}{k-t-2i}.
\]
By Lemma \ref{lem-3.1} (ii) and Lemma \ref{lem-F20},
\[
    |\hf_3| \leq \binom{t+3}{3}\binom{n-t-3}{k-t-6}.
\]
By Lemma \ref{lem-3.1} (iii) and Theorem \ref{thm-EKR},
\[
|\hf_2|-|\thf_2([t+1])|=\sum_{T\in \binom{[t+3]}{t+1}\setminus \{[t+1]\}}|\thf_2(T)| \leq \left(\binom{t+3}{2}-1\right)\binom{n-t-7}{k-t-5}.
\]
 Since $\hf$ is 2-wise $(t+3)$-intersecting, $\thf_1([t+3]\setminus \{i\})$ is intersecting for $i\in [t+3]$. By the Erd\H{o}s-Ko-Rado Theorem,
\[
|\hf_1| \leq (t+3)\binom{n-t-4}{k-t-3}.
\]
Clearly $|\hf_0|\leq \binom{n-t-3}{k-t-3}$.
Thus,
\[
\begin{aligned}
|\hf|-|\thf_{2}([t+1])|\leq{}&\sum_{3\leq i\leq t+3}\binom{t+3}{i}\binom{n-t-3}{k-t-2i}
+\binom{t+3}{2}\binom{n-t-6}{k-t-5}\\
&{}+(t+3)\binom{n-t-4}{k-t-3}+\binom{n-t-3}{k-t-3}.
\end{aligned}
\]

For $i\geq4$, by \eqref{ineq-3.0} we have
\[
\binom{t+3}{i}\binom{n-t-3}{k-t-2i}
\leq\frac{t\alpha^2}{4}\binom{t+3}{i-1}\binom{n-t-3}{k-t-2i+2}.
\]
Summing this geometric bound and using \eqref{ineq-3.0} four more times, we obtain
\begin{align*}
\sum_{3\leq i\leq t+3}\binom{t+3}{i}\binom{n-t-3}{k-t-2i}
\leq\frac{2\theta^2(t+1)(t+2)}{3(4(t+3)-\theta t)}\binom{n-t-3}{k-t-2}.
\end{align*}
Moreover,
\begin{align*}
&\quad\binom{t+3}{2}\binom{n-t-6}{k-t-5}+(t+3)\binom{n-t-4}{k-t-3}\leq\left(\frac{\theta(t+2)\alpha}{2(1+\alpha)^3}+\frac{(t+3)\alpha}{1+\alpha}\right)\binom{n-t-3}{k-t-2}.
\end{align*}
Thus,
\begin{align*}
|\hf|-|\thf_2([t+1])|\leq{}&\left(\frac{2\theta^2(t+1)(t+2)}{3(4(t+3)-\theta t)}
+\frac{\theta(t+2)\alpha}{2(1+\alpha)^3}+\frac{(t+3)\alpha}{1+\alpha}\right)\binom{n-t-3}{k-t-2}\\
&+\binom{n-t-3}{k-t-3}.
\end{align*}
If $|\thf_2([t+1])|\leq2\binom{n-t-3}{k-t-2}$, then for $t\geq45$,
\begin{align*}
|\hf|\leq{}&\left(\frac{2\theta^2(t+1)(t+2)}{3(4(t+3)-\theta t)}
+\frac{\theta(t+2)\alpha}{2(1+\alpha)^3}+\frac{(t+3)\alpha}{1+\alpha}+2\right)
\binom{n-t-3}{k-t-2}+\binom{n-t-3}{k-t-3}\\
<&(t+3)\binom{n-t-3}{k-t-2}+\binom{n-t-3}{k-t-3}=|\ha(n,k,t+3)|,
\end{align*}
a contradiction. Thus $|\thf_2([t+1])|>2\binom{n-t-3}{k-t-2}$.
\end{proof}

\begin{proof}[Proof of Theorem \ref{thm-main2}]
The case $k=t+2$ follows from the argument in the beginning of the proof of Theorem \ref{thm-main}. Hence we assume $k\geq t+3$.
Let $\hf\subset \binom{[n]}{k}$ be a non-trivial $3$-wise $t$-intersecting family with $|\hf|=h(n,k,3,t)$. Since $n\geq2k-t$, by Corollary \ref{cor-2.1}, we may assume that $\hf$ is shifted. Let $\theta=1.663$ and $\alpha=\sqrt{\frac{\theta}{t+3}}$.
For $t\geq45$, one can check that
\[
\frac{4(t+3)}{(\sqrt{4t+9}-3)^2}<\theta.
\]
Consequently,
\[
n\geq\frac{\sqrt{4t+9}-1}{2}k\geq\left(1+\frac1\alpha\right)k\geq k+\frac{k-t}{\alpha}.
\]
Thus $k-t\leq\alpha(n-k)$.

To obtain a contradiction, assume $|\hf|>\max\{|\ha(n,k,t+3)|,|\hb(n,k,t+3)|\}$. Then Lemma \ref{lem-key} applies.
Since $\hf$ is non-trivial, there is a member in $\hf$ that does not contain $[t]$ as a subset. Then by shiftedness $[k+1]\setminus\{t\}\in\hf$. Since $\hf$ is 2-wise $(t+1)$-intersecting, $[t+1]\cup\{k+2,k+3,\ldots,2k-t\}\notin\hf$. By shiftedness, we infer that
\begin{align}\label{ineq-5.1}
G\cap[t+2,k+1]\neq\emptyset\mbox{ for any }G\in\hf([t+1],[t+1]).
\end{align}
If $k\leq t+4$, every member of $\thf_2([t+1])$ intersects $[t+4,k+1]\subseteq\{t+4,t+5\}$. Consequently,
\[
|\thf_2([t+1])|\leq\binom{n-t-4}{k-t-2}+\binom{n-t-5}{k-t-2}<2\binom{n-t-3}{k-t-2},
\]
contrary to \eqref{ineq-3.2}. Hence $k\geq t+5$.

If $(t+6,t+8,\ldots,2k-t+2)\notin\thf_2([t+1])$, applying Proposition \ref{prop-F78} (ii) with $p=3$ to the $(k-t-1)$-uniform family $\thf_2([t+1])$, we obtain that
\begin{align*}
|\thf_2([t+1])|
&\leq\binom{n-t-3}{k-t-1}-\binom{n-t-4}{k-t-1}+\binom{n-t-4}{k-t-2}
=2\binom{n-t-4}{k-t-2},
\end{align*}
contradicting \eqref{ineq-3.2}. Thus we may assume $(t+6,t+8,\ldots,2k-t+2)\in\thf_2([t+1])$.

Let $j$ be the maximum integer satisfying $(t+3+j,t+5+j,\ldots,2k-t+j-1)\in\thf_2([t+1])$. Then $j\geq3$. By \eqref{ineq-5.1}, we have $3\leq j\leq k-t-2$.

\begin{claim}
    If $j=k-t-2$ then $\hf\subset \hb(n,k,t+3)$.
\end{claim}

\begin{proof}
By Lemma \ref{lem-3.2} (iv), $\thf_{2}([t+1])$, $\thf_1([t+3]\setminus \{i\})$ are cross-intersecting for any $i\in [t+1]$. Since $(k+1,k+3,\ldots, 3k-2t-3)\in \thf_{2}([t+1])$, by shiftedness $\thf_1([t+3]\setminus \{i\})\subset \{[t+4,k+1]\}$ for any $i\in [t+1]$.
By \eqref{ineq-5.1}, we are left to show that $[t+1]\subset F$ for any $F\in \hf\setminus (\hf_0\cup \hf_1)$. If  there exists $F\in \hf\setminus (\hf_0\cup \hf_1)$ with $[t+1]\not\subset F$, then by shiftedness $[k+2]\setminus \{t+1,k+1\}\in \hf$. However, $([k+2]\setminus \{t+1,k+1\})\cap ([t+1]\cup \{k+1,k+3,\ldots,3k-2t-3\}) =[t]$, contradicting the fact that $\hf$ is 2-wise $(t+1)$-intersecting. Thus $[t+1]\subset F$ for any  $F\in \hf\setminus (\hf_0\cup \hf_1)$ and $\hf\subset \hb(n,k,t+3)$ follows.
\end{proof}

Thus we may assume $3\leq j\leq k-t-3$.
By the maximality of $j$, we infer that
\[
(t+3+j+1,t+5+j+1,\ldots, 2k-t+j)\notin \thf_{2}([t+1]).
\]
By applying Proposition \ref{prop-F78} (ii) with $p=j+1$, we obtain that
\[
|\thf_{2}([t+1])| \leq \binom{n-t-3}{k-t-1}- \binom{n-t-j-2}{k-t-1}+\binom{n-t-j-2}{k-t-2}.
\]
Let $\hb:=\hb(n,k,t+3)$.
By \eqref{ineq-5.1}, $\hf([t+1],[t+1])\subset \hb([t+1],[t+1])$. Since
$|\tilde{\hb}_2([t+1])|=\binom{n-t-3}{k-t-1}-\binom{n-k-1}{k-t-1}$,
we infer that
\begin{align}\label{ineq-neq3.4}
|\hb\setminus\hf| \geq  \binom{n-t-j-2}{k-t-1}- \binom{n-k-1}{k-t-1}-\binom{n-t-j-2}{k-t-2}.
\end{align}

On the other hand, by $\hf([t+1],[t+1])\subset \hb([t+1],[t+1])$ we have
\begin{align}\label{ineq-neq3.5}
|\hf\setminus \hb|\leq |\{F\in \hf\colon [t+1]\not\subset F\}|.
\end{align}
By Lemma \ref{lem-3.2} (iv),  $\thf_1([t+3]\setminus \{i\}),\thf_{2}([t+1])$ are cross-intersecting for any $i\in [t+1]$. It follows that
\[
[t+4,t+1+j]\cup (t+2+j, t+4+j,\ldots, 2k-t-j)\notin \thf_1([t+3]\setminus \{i\}).
\]
By applying
 Proposition \ref{prop-F78} (i) with $p=j-1$,
\begin{align}\label{ineq-new5.1}
\sum_{1\leq i\leq t+1}|\thf_1([t+3]\setminus \{i\})| \leq (t+1)\binom{n-t-3}{k-t-j-2}.
\end{align}

If a lower index in the bounds below is negative, the corresponding family is empty: otherwise shiftedness implies $[t+4,k+i]\in\thf_i(T)$, whose intersection with $(t+3+j,t+5+j,\ldots,2k-t+j-1)$ is too small for the corresponding part of Lemma \ref{lem-3.2}.

By Lemma \ref{lem-3.2} (i), $\thf_{2}([t+1]), \thf_i(T)$ are cross $(2i-1)$-intersecting for $i=4,\ldots,t+3$ and for any $T\in \binom{[t+3]}{t+3-i}$. Since $(t+3+j,t+5+j,\ldots, 2k-t+j-1)\in \thf_{2}([t+1])$, we infer that  for any $T\in \binom{[t+3]}{t+3-i}$,
\[
[t+4,t+j+4i-3]\cup (t+j+4i-2,t+j+4i,\ldots ) \notin \thf_i(T).
\]
By applying Proposition \ref{prop-F78} (i) with $p=j+4i-5$,
\[
|\thf_i(T)| \leq \binom{n-t-3}{k-(t+3-i)-(j+4i-5)-1} = \binom{n-t-3}{k-t-j-3i+1}.
\]
Therefore,
\begin{align}\label{ineq-new5.2}
\sum_{4\leq i\leq t+3} |\hf_i| \leq \sum_{4\leq i\leq t+3} \binom{t+3}{i} \binom{n-t-3}{k-t-j-3i+1}.
\end{align}

By Lemma \ref{lem-3.2} (ii) for any $T\in \binom{[t+3]}{t}$, $\thf_{2}([t+1])$, $\thf_3(T)$ are cross 3-intersecting if $T=[t]$ and also $\thf_{2}([t+1]), \thf_3(T)$ are cross $5$-intersecting if $T\neq [t]$. Consequently,
\begin{align*}
    &[t+4,t+j+5]\cup (t+j+6,t+j+8,\ldots ) \notin \thf_3([t]),\\[3pt]
    &[t+4,t+j+9]\cup (t+j+10,t+j+12,\ldots ) \notin \thf_3(T), \mbox{ for }T\neq [t].
\end{align*}
By using Proposition \ref{prop-F78} (i) with $p=j+3$ for $T=[t]$ and $p=j+7$ for $T\neq[t]$,
\begin{align}\label{ineq-new5.3-1}
    |\hf_3|\leq  \left(\binom{t+3}{3}-1\right)\binom{n-t-3}{k-t-j-8} +\binom{n-t-3}{k-t-j-4}
\end{align}

By Lemma \ref{lem-3.2} (iii), for any $T\in \binom{[t+3]}{t+1}$, $\thf_{2}([t+1])$, $\thf_{2}(T)$ are cross 2-intersecting if $T\cap [t+1]=[t]$ and  $\thf_{2}([t+1]),\thf_{2}(T)$ are cross 3-intersecting if $[t]\not\subset T$. Then
\begin{align*}
    &[t+4,t+j+3]\cup (t+j+4,t+j+6,\ldots ) \notin \thf_2(T), \mbox{ for }T\cap [t+1]=[t], \\[3pt]
    &[t+4,t+j+5]\cup (t+j+6,t+j+8,\ldots ) \notin \thf_2(T), \mbox{ for }[t]\not\subset T.
\end{align*}
By using
 Proposition \ref{prop-F78} (i) with $p=j+1$ and $p=j+3$, respectively,
\begin{align}\label{ineq-new5.3-2}
    \sum_{T\in \binom{[t+3]}{t+1}\setminus \{[t+1]\}}|\thf_{2}(T)|\leq  2\binom{n-t-3}{k-t-j-3}+\left(\binom{t+3}{2}-3\right)\binom{n-t-3}{k-t-j-5}.
\end{align}
Adding \eqref{ineq-new5.1}, \eqref{ineq-new5.2}, \eqref{ineq-new5.3-1} and \eqref{ineq-new5.3-2},
\begin{align*}
 &\quad |\hf\setminus \hb| \\[3pt]
 &\leq \sum_{1\leq i\leq t+1}|\thf_1([t+3]\setminus \{i\})| +\sum_{3\leq i\leq t+3} |\hf_i| +\sum_{T\in \binom{[t+3]}{t+1}\setminus \{[t+1]\}}|\thf_{2}(T)|\\[3pt]
 &\leq  (t+1)\binom{n-t-3}{k-t-j-2}+\sum_{2\leq i\leq t+3} \binom{t+3}{i} \binom{n-t-3}{k-t-j-3i+1} +2\binom{n-t-3}{k-t-j-3}\\[3pt]
 &\qquad +\binom{n-t-3}{k-t-j-4}.
 \end{align*}
Note that $\frac{\theta^{3/2}}{\sqrt{t+3}} <\frac{1}{3}$ for $t\geq45$. For $i\geq 3$,
 \begin{align*}
\binom{t+3}{i} \binom{n-t-3}{k-t-j-3i+1} &\leq \binom{t+3}{i-1} \frac{t+4-i}{i} \alpha^3\binom{n-t-3}{k-t-j-3(i-1)+1}\\[3pt]
&\leq  \frac{\theta^{3/2}}{3\sqrt{t+3}} \binom{t+3}{i-1}\binom{n-t-3}{k-t-j-3(i-1)+1}\\[3pt]
&\leq\frac{1}{9}\binom{t+3}{i-1}\binom{n-t-3}{k-t-j-3(i-1)+1}.
 \end{align*}
Using \eqref{ineq-3.0} we infer that
\begin{align*}
\sum_{2\leq i\leq t+3} \binom{t+3}{i} \binom{n-t-3}{k-t-j-3i+1} &\leq \frac{9}{8}\binom{t+3}{2} \binom{n-t-3}{k-t-j-5} \\[3pt]
&\leq \frac{9(t+3)(t+2)}{16} \alpha^3\binom{n-t-3}{k-t-j-2}\\[3pt]
&\leq \frac{9\theta^{3/2}}{16}\sqrt{t+3}\binom{n-t-3}{k-t-j-2}.
\end{align*}
Moreover,  by $\alpha<\frac15$ for $t\geq45$,
\[
2\binom{n-t-3}{k-t-j-3}+\binom{n-t-3}{k-t-j-4} \leq (2\alpha +\alpha^2)\binom{n-t-3}{k-t-j-2} <\frac{1}{2}\binom{n-t-3}{k-t-j-2}.
\]
Since $\sqrt{t+3}/t$ is a decreasing function of $t$, checking $t=45$ shows that, for every $t\geq45$,
\begin{align*}
 |\hf\setminus \hb| &\leq \left(t+1+\frac{9\theta^{3/2}}{16}\sqrt{t+3}+\frac{1}{2}\right)\binom{n-t-3}{k-t-j-2}<\frac{11t}{9}\binom{n-t-3}{k-t-j-2}.
\end{align*}
In view of \eqref{ineq-neq3.4} and \eqref{ineq-neq3.5}, we are left to show that
\begin{align}\label{ineq-5.2}
\frac{11}{9}t\binom{n-t-3}{k-t-j-2}<\binom{n-t-j-2}{k-t-1}- \binom{n-k-1}{k-t-1}-\binom{n-t-j-2}{k-t-2}.
\end{align}
Since $j\leq k-t-3$, we have
\[
\binom{n-t-j-2}{k-t-1}- \binom{n-k-1}{k-t-1} \geq \binom{n-t-j-3}{k-t-2}+\binom{n-t-j-4}{k-t-2}.
\]
Then by $t+j+3\leq k$,
\[
\frac{n-k-j}{n-t-j-2} >\frac{n-k-j-1}{n-t-j-3} =1-\frac{k-t-2}{n-t-j-3}>1-\frac{k-t}{n-(t+j+3)}\geq 1-\frac{k-t}{n-k}\geq 1-\alpha.
\]
 Thus,
\begin{align*}
 &\quad \  \binom{n-t-j-3}{k-t-2}+\frac{2}{7}\binom{n-t-j-4}{k-t-2}\\[3pt] & = \left(\frac{n-k-j}{n-t-j-2}+ \frac{2}{7}\frac{(n-k-j)(n-k-j-1)}{(n-t-j-2)(n-t-j-3)}\right)\binom{n-t-j-2}{k-t-2}\\[3pt]
  &> \left(1-\alpha+ \frac{2}{7}\left(1-\alpha\right)^2\right)\binom{n-t-j-2}{k-t-2}.
\end{align*}
For $t\geq45$, we have $\alpha<\frac3{16}$ and $1-\alpha+\frac27(1-\alpha)^2>1$. It follows that
\[
\binom{n-t-j-3}{k-t-2}+\frac{2}{7}\binom{n-t-j-4}{k-t-2}>\binom{n-t-j-2}{k-t-2}.
\]
Thus, to show \eqref{ineq-5.2} it suffices to show that
\[
\frac{11}{9}t\binom{n-t-3}{k-t-j-2}<\frac{5}{7}\binom{n-t-j-4}{k-t-2}.
\]
Note that
\[
\frac{77}{45}t\binom{n-t-3}{k-t-j-2} \leq \frac{77}{45}t \alpha^j\binom{n-t-3}{k-t-2}.
\]
For $4\leq x\leq7$, we have
\[
\frac{\binom{n-t-x}{k-t-2}}{\binom{n-t-x+1}{k-t-2}}
=\frac{n-k-x+3}{n-t-x+1}\geq\frac1{1+\alpha},
\]
because $\alpha(n-k)-(k-t)+2-\alpha(x-3)\geq2-4\alpha>0$.

Since $j\leq k-t-3$, for $8\leq x\leq j+4\leq k-t+1$,
\[
\frac{\binom{n-t-x}{k-t-2}}{\binom{n-t-x+1}{k-t-2}}
=\frac{n-k-x+3}{n-t-x+1}\geq\frac{n-k-(k-t)+2}{n-t-(k-t)}
\geq1-\frac{k-t}{n-k}\geq1-\alpha.
\]
Consequently,
\[
\binom{n-t-j-4}{k-t-2}
\geq\frac{(1-\alpha)^{j-3}}{(1+\alpha)^4}\binom{n-t-3}{k-t-2}.
\]
Thus it suffices to prove
\[
\frac{(1-\alpha)^{j-3}}{(1+\alpha)^4}\geq\frac{77}{45}t\alpha^j.
\]
At $t=45$, direct calculation shows $(1+\alpha)^{-4}>\frac{77}{45}t\alpha^3$.
As $t$ increases, $(1+\alpha)^{-4}$ increases and
$t\alpha^3=\theta^{3/2}t/(t+3)^{3/2}$ decreases for $t>6$.
Since $1-\alpha\geq\alpha$, the required inequality follows for every $t\geq45$ and $j\geq3$.
This concludes the proof of the theorem.
\end{proof}

\section{Concluding remarks}

\begin{definition}
For $t\geq 1$, $r\geq 2$, $s\geq 0$ and $n\geq t+rs$ define
\[
\hb_s(n,r,t) =\{B\subset [n]\colon |[t+rs]\setminus B|\leq s\}.
\]
\end{definition}

These families are clearly $r$-wise $t$-intersecting. Half a century ago the first author conjectured that these are the largest non-uniform $r$-wise $t$-intersecting families:

\begin{conjecture}[\cite{F76}]
Suppose that $\hh\subset 2^{[n]}$ is $r$-wise $t$-intersecting, then
\begin{align}\label{ineq-6.1}
|\hh| \leq \max_{0\leq s\leq \frac{n-t}{r}} |\hb_s(n,r,t)|.
\end{align}
\end{conjecture}

Unfortunately \eqref{ineq-6.1} is still wide open in general. However the ``Erd\H{o}s-Ko-Rado" case is known to be true (cf. \cite{F91}, \cite{F19}). Note that $\hb_0(n,r,t)=\{B\subset [n]\colon [t]\subset B\}$.

\vspace{5pt}
{\noindent\bf Exact EKR Theorem for non-uniform families (\cite{F91}, \cite{F19})}
Suppose that $\hh\subset 2^{[n]}$ is $r$-wise $t$-intersecting, $1\leq t\leq 2^r-r-1$ and $r\geq 3$. Then $|\hh|\leq 2^{n-t}$. Moreover, equality holds if and only if $\hh\cong \hb_0(n,r,t)$ or $t=2^{r}-r-1$, $n\geq t+r$ and $\hh\cong \hb_1(n,r,t)$.
\vspace{5pt}

It is possible to restate this result in the following form.

\begin{theorem}[\cite{F91}, \cite{F19}]
Suppose that $\hh\subset 2^{[n]}$ is $r$-wise $t$-intersecting and $|\hb_1(n,r,t)|\leq  2^{n-t}$. Then $|\hh|\leq 2^{n-t}$ as well.
\end{theorem}

Let us state the uniform analogue of this result using the notation $\hh^{(k)}=\{H\in \hh\colon |H|=k\}$.

\begin{conjecture}[Exact EKR]\label{conj-6.4}
Let $\hf\subset \binom{[n]}{k}$ be $r$-wise $t$-intersecting, $(r-1)(n-k)>k-t$ and suppose that $|\hb_1^{(k)}(n,r,t)|\leq  \binom{n-t}{k-t}$. Then
\begin{align}\label{ineq-6.2}
|\hf| \leq  \binom{n-t}{k-t}.
\end{align}
\end{conjecture}

During the paper, we used the slightly simpler notation $\ha(n,k,r+t)$ for $\hb_1^{(k)}(n,r,t)$.

In Theorem \ref{thm-main0} we showed that Conjecture \ref{conj-6.4} holds for the case $r=3$ and $k\geq t+2\geq44$. In an earlier paper \cite{FW25} we proved \eqref{ineq-6.2} for $n\geq \left(2.5 t\right)^{\frac{1}{r-1}}(k-t)+k$. The case $t=1$ is relatively simple but already for $t=2$, $r=3$ it is not known whether Conjecture \ref{conj-6.4} is true.

Solving the problem in full generality under the additional restriction that $\hf$ is non-trivial appears to be hopeless.

\end{document}